%% file: bar-cobar-2011-03-01.tex
\documentclass[11pt]{amsart}

\usepackage{amssymb}
\usepackage{graphicx}
\usepackage{pb-diagram,pb-xy} 
\usepackage[cmtip,arrow]{xy} 
\usepackage[margin=1in]{geometry}  
\usepackage[dvips]{color} 
\usepackage{hyperref} 

\numberwithin{equation}{section}

\theoremstyle{plain}
\newtheorem{theorem}[equation]{Theorem}
\newtheorem{thm}[equation]{Theorem}
\newtheorem{lemma}[equation]{Lemma}

\newtheorem{cor}[equation]{Corollary}

\newtheorem{proposition}[equation]{Proposition}
\newtheorem{conjecture}[equation]{Conjecture}
\newtheorem{prop}[equation]{Proposition}

\newtheorem*{claim*}{Claim}
\newtheorem*{thm*}{Theorem}

\theoremstyle{definition}
\newtheorem{definition}[equation]{Definition}
\newtheorem{informal-definition}[equation]{Informal Definition}

\newtheorem{remark}[equation]{Remark}

\newtheorem{example}[equation]{Example}
\newtheorem{examples}[equation]{Examples}

\newcommand{\isom}{\cong}                       
\newcommand{\homeq}{\simeq}                     
\newcommand{\smsh}{\wedge}                      
\newcommand{\union}{\cup}                       

\newcommand{\Smsh}{\bigwedge}                   
\newcommand{\Wdge}{\bigvee}                     

\newcommand{\Q}{\mathbb{Q}}                     

\newcommand{\cat}[1]{\mathcal{#1}}              

\DeclareMathOperator*{\hocolim}{hocolim}
\DeclareMathOperator*{\colim}{colim}
\DeclareMathOperator*{\holim}{holim}
\newcommand{\Map}{\operatorname{Map} }

\newcommand{\Hom}{\operatorname{Hom} }
\newcommand{\Tot}{\operatorname{Tot} }

\newcommand{\sset}{\mathsf{sSet}_*}            

\newcommand{\spectra}{{\mathsf{Spec}}}              

\newcommand{\weq}{\; \tilde{\longrightarrow} \;}      
\newcommand{\lweq}{\; \tilde{\longleftarrow} \;}      
\newcommand{\epi}{\twoheadrightarrow}           

\newcommand{\dual}{\mathbb{D}}                  
\newcommand{\der}{\partial}                     

\newcommand{\ord}[1]{$#1$\textsuperscript{th}}

\begin{document}

\title{Bar-cobar duality for operads in stable homotopy theory}
\author{Michael Ching}
\thanks{This work was supported in part by NSF Grant DMS-0968221.}

\begin{abstract}
We extend bar-cobar duality, defined for operads of chain complexes by Getzler and Jones, to operads of spectra in the sense of stable homotopy theory. Our main result is the existence of a Quillen equivalence between the category of reduced operads of spectra (with the projective model structure) and a new model for the homotopy theory of cooperads of spectra. The crucial construction is of a weak equivalence of operads between the Boardman-Vogt $W$-construction for an operad $P$, and the cobar-bar construction on $P$. This weak equivalence generalizes a theorem of Berger and Moerdijk that says the $W$- and cobar-bar constructions are isomorphic for operads of chain complexes.

Our model for the homotopy theory of cooperads is based on `pre-cooperads'. These can be viewed as cooperads in which the structure maps are zigzags of maps of spectra that satisfy coherence conditions. Our model structure on pre-cooperads is such that every object is weakly equivalent to an actual cooperad, and weak equivalences between cooperads are detected in the underlying symmetric sequences.

We also interpret our results in terms of a `derived Koszul dual' for operads of spectra, which is analogous to the Ginzburg-Kapranov dg-dual. We show that the double derived Koszul dual of an operad $P$ is equivalent to $P$ (under some finiteness hypotheses) and that the derived Koszul construction preserves homotopy colimits, finite homotopy limits and derived mapping spaces for operads.
\end{abstract}

\maketitle

\thispagestyle{empty}

In an influential paper \cite{ginzburg/kapranov:1994}, Ginzburg and Kapranov describe various types of duality for operads of differential graded vector spaces. For such an operad $P$, they define a `dg-dual' $\mathbf{D}(P)$ such that the double dg-dual $\mathbf{D}(\mathbf{D}(P))$ is quasi-isomorphic to the original operad $P$ (under the condition that the terms in the operad $P$ are finite-dimensional vector spaces). The same authors define a `Koszul dual' $P^{!}$ for quadratic operads and show that under certain conditions, $P^{!}$ is quasi-isomorphic to $\mathbf{D}(P)$. In this case, we can think of $\mathbf{D}(P^{!})$ as a `small' cofibrant replacement for $P$. Algebras over the operad $\mathbf{D}(P^{!})$ play the role of `$P$-algebras up to homotopy'.

These definitions generalize the bar-cobar duality for algebras and coalgebras discovered by Moore \cite{moore:1971}, as well as Priddy's notion of Koszul duality for algebras \cite{priddy:1970}. They also illuminate the relationship between Quillen's models for rational homotopy theory \cite{quillen:1969}. Getzler and Jones, in \cite{getzler/jones:1994}, restated the Ginzburg-Kapranov results in a way that makes the analogy with Moore's work more striking. They describe a bar construction that takes operads to cooperads and, dually, a cobar construction that takes cooperads to operads. These functors are adjoint and determine an equivalence of homotopy categories between augmented operads and connected coaugmented cooperads of differential graded vector spaces \cite[2.17]{getzler/jones:1994}.

The main aim of this paper is to establish an analogous equivalence for operads and cooperads of spectra. In previous work \cite{ching:2005a}, we described analogues of the bar and cobar constructions for operads in any symmetric monoidal category suitably enriched over topological spaces. We prove the following result.

\begin{thm*}
There is a Quillen equivalence of the form
\[ \mathsf{Operad} \underset{C}{\overset{\mathbb{B}}{\rightleftarrows}} \mathsf{PreCooperad} \]
where:
\begin{itemize}
  \item $\mathsf{Operad}$ is the category of reduced operads of spectra with the projective model structure;
  \item $\mathsf{PreCooperad}$ is a new model category that contains the reduced cooperads as a full subcategory, and for which every object is weakly equivalent to a cooperad;
  \item $C$ is an extension of the cobar construction to all pre-cooperads;
  \item the bar construction is weakly equivalent to the left derived functor of $\mathbb{B}$.
\end{itemize}
\end{thm*}
The cofibrant-fibrant objects in $\mathsf{PreCooperad}$ can be thought of as cooperads in which the structure maps involve inverse weak equivalences in a coherent way. We call these `quasi-cooperads'. The homotopy category of $\mathsf{PreCooperad}$ can be identified with that of these `quasi-cooperads' with weak equivalences detected in the underlying symmetric sequences.

Note that in contrast to the adjunction studied by Getzler and Jones, the cobar construction is the \emph{right} adjoint of the Quillen pair.

The key step in the proof that $\mathbb{B}$ and $C$ form a Quillen equivalence can be stated without mentioning pre-cooperads and is of interest in its own right. This is the construction of a weak equivalence of operads
\[ \theta: WP \weq CBP \]
where $WP$ denotes the Boardman-Vogt $W$-construction \cite{boardman/vogt:1973} for the operad $P$, and $B$ and $C$ are the bar/cobar constructions of \cite{ching:2005a}. This generalizes a result of Berger and Moerdijk \cite{berger/moerdijk:2006} that says $WP$ and $CBP$ are isomorphic for an operad $P$ of chain complexes. The definition of $\theta$ involves detailed consideration of certain simplicial sets associated to categories of trees. (See \ref{def:W-CB}.) We prove that $\theta$ is a weak equivalence by reducing to the case of trivial operads for which the bar-cobar construction can be explicitly evaluated.

We also use Spanier-Whitehead duality to state some of our results purely in terms of operads. For an operad $P$ of spectra, we define a new operad $KP$ to be the Spanier-Whitehead dual of the cooperad $BP$. Thus $KP$ is the analogue of Ginzburg and Kapranov's dg-dual $\mathbf{D}(P)$. Using the map $\theta$ we show that, if the terms of $P$ are weakly equivalent to finite cell spectra, there is a natural equivalence of operads
\[ K(K(P)) \homeq P. \]
We also discuss the effect of the functor $K$ on homotopy limits and colimits, and on mapping objects for operads. In particular, we prove that, for termwise-finite operads $P$ and $P'$,
\[ \widetilde{\Hom}_{\mathsf{Operad}}(P,P') \homeq \widetilde{\Hom}_{\mathsf{Operad}}(KP',KP). \]
where $\widetilde{\Hom}_{\mathsf{Operad}}(-,-)$ denotes the derived mapping space for operads of spectra.

Our main example of this theory concerns Goodwillie's calculus of functors \cite{goodwillie:2003}. In \cite{arone/ching:2009} we proved that
\[ K(\der^*(\Sigma^\infty \Omega^\infty)) \homeq \der_*I \]
where $\der^*(\Sigma^\infty \Omega^\infty)$ is an operad formed by the Spanier-Whitehead duals of the Goodwillie derivatives of the functor
\[ \Sigma^\infty \Omega^\infty: \spectra \to \spectra \]
and $\der_*I$ is an operad formed by the derivatives of the identity functor on based spaces. We now deduce also that
\[ K(\der_*I) \homeq \der^*(\Sigma^\infty \Omega^\infty). \]
In \S\ref{sec:examples} we conjecture that a similar pair of dual operads exists for the identity functor on other categories in which one can do Goodwillie calculus.

One note: we work in this paper only with operads of spectra, where by spectra we really mean the S-modules of EKMM \cite{elmendorf/kriz/mandell/may:1997}. However, many of our constructions can be made in any pointed symmetric monoidal model category $\cat{C}$ suitably enriched over simplicial sets. If the projective model structure on operads in $\cat{C}$ exists then there is a Quillen adjunction as in the Theorem. If the model structure on $\cat{C}$ is \emph{stable} then this adjunction is a Quillen equivalence. In particular, our main Theorem applies also to operads of symmetric or orthogonal spectra.

\subsection*{Summary of the paper}

In \S\ref{sec:$W$-construction} we fix some of our notation for operads and describe the Boardman-Vogt $W$-construction. The key step of the paper appears in \S\ref{sec:bar} where we construct the map of operads $\theta: WP \to CBP$ and show, in Theorem \ref{thm:W-CB}, that it is a weak equivalence.

We prove our main result in \S\ref{sec:precooperads}. We introduce the category of pre-cooperads and its model structure. We describe the left adjoint $\mathbb{B}$ of the cobar construction and show that $\mathbb{B}$ and $C$ form a Quillen equivalence. We then identify the cofibrant-fibrant objects in this model structure with `quasi-cooperads' and show that weak equivalences between quasi-cooperads are detected termwise on the underlying symmetric sequences. In \S\ref{sec:koszul} we use Spanier-Whitehead duality to reinterpret our results in terms of a `derived Koszul dual' functor from operads to operads.

In \S\ref{sec:examples} we discuss the example of bar-cobar duality that arises in Goodwillie calculus and make various conjectures for other examples motivated by Koszul duality results on the algebraic level. These include conjectures for the duals of the stable little $n$-discs operads, and for the operad formed by the Deligne-Mumford compactifications of moduli spaces of Riemann surfaces with marked points.

The final section \S\ref{sec:P-CBP} concerns a generalization of a result of \cite{arone/ching:2009} needed in the proof that the map $\theta: WP \to CBP$ is a weak equivalence.

\subsection*{Acknowledgements}

This project started essentially when I was a graduate student at MIT. Both then and since I have benefitted greatly from conversations with Haynes Miller on these and other topics. Motivation for writing this paper came from joint work with Greg Arone, and results from that work significantly influenced what is here. I'm grateful to Greg for all the help and advice he has provided over the past few years. I'd like to thank the organizers and participants of the Workshop on Operads and Homotopy Theory in Lille in August 2010 for the opportunity to present some of the results of this paper, and for some useful feedback. Finally, thanks to a referee for corrections and suggestions for improvement.

\newpage

\section{Operads of spectra and the $W$-construction} \label{sec:$W$-construction}

We work in the category of $S$-modules described by EKMM \cite{elmendorf/kriz/mandell/may:1997}, which we denote by $\spectra$. We refer to objects in $\spectra$ as \emph{spectra} rather than $S$-modules. Much of this theory could be developed for other models of stable homotopy theory, and in a wider context of simplicial symmetric monoidal model categories, or indeed other types of enriched model categories, such as over chain complexes. It is convenient for us, however, that all objects of $\spectra$ are fibrant, so we restrict to the EKMM case.\footnote{Specifically, we use this condition to show that the cobar construction preserves weak equivalences between all cooperads in Proposition \ref{prop:bar-cobar-homotopy}. This allows us to form a homotopy-invariant cobar-bar construction $CBP$ for an operad $P$. The machinery of \S\ref{sec:precooperads} could be used to circumvent this as it allows us to take a fibrant replacement for $BP$ (in the $C$-model structure on pre-cooperads described in Proposition \ref{prop:model-precooperad}).}

In this section, we recall the definition of an operad in the category $\spectra$, and describe the $W$-construction of Boardman-Vogt \cite{boardman/vogt:1973} in this context.

\begin{definition}[Symmetric sequences] \label{def:symseq}
Let $\mathsf{\Sigma}$ denote the category of nonempty finite sets and bijections. A \emph{symmetric sequence} $A$, in $\spectra$, is a functor $A: \mathsf{\Sigma} \to \spectra$. We denote the category of such symmetric sequences and their natural transformations by $\spectra^{\mathsf{\Sigma}}$.

Equivalently, one can view a symmetric sequence $A$ as a sequence $A(1), A(2),\dots$ of spectra together with a (right) action of the symmetric group $\Sigma_n$ on $A(n)$, for each $n$. The connection between these two viewpoints is that $A(n)$ represents the value of $A$ on the finite set $\{1,\dots,n\}$.
\end{definition}

\begin{definition}[Operads] \label{def:operad}
An \emph{operad} $P$ in $\spectra$ consists of a symmetric sequence $P$ together with \emph{composition maps}, for each finite disjoint union of nonempty finite sets $I = \coprod_{j \in J} I_j$:
\[ P(f): P(J) \smsh \Smsh_{j \in J} P(I_j) \to P(I) \]
and a \emph{unit map} $\eta: S \to P(1)$ where $S$ is the sphere spectrum (i.e. the unit object for the smash product) and $P(1)$ denotes the value of $P$ on the one-element set $\{1\}$. These maps satisfy standard naturality, associativity and unitivity conditions (see \cite[2.2]{ching:2005a}).

An operad is \emph{reduced} if the unit map $\eta$ is an isomorphism between $S$ and $P(1)$. In this paper we consider only reduced operads. A morphism $P \to P'$ of reduced operads consists of a natural transformation between the symmetric sequences $P$ and $P'$ that commutes with the composition and unit maps. We thus obtain a category $\mathsf{Operad}$ of reduced operads in $\spectra$.

In \cite[Appendix]{arone/ching:2009}, we showed that $\mathsf{Operad}$ is enriched, tensored and cotensored over the category $\sset$ of pointed simplicial sets, and that it has a cofibrantly generated simplicial model structure, in which weak equivalences and fibrations are defined termwise. We refer to this as the \emph{projective model structure} on the category of operads.
\end{definition}

\begin{examples} \label{ex:operads}
(1) If $\mathcal{P}$ is an operad of unbased topological spaces (with respect to cartesian product), then we get an operad $\Sigma^\infty_+\mathcal{P}$ in $\spectra$ from
\[ (\Sigma^\infty_+\mathcal{P})(I) := \Sigma^\infty \mathcal{P}(I)_+ \]
with composition maps determined from those of $\mathcal{P}$ by the isomorphisms
\[ \Sigma^\infty (X \times Y)_+ \isom \Sigma^\infty X_+ \smsh \Sigma^\infty Y_+. \]
For example, we have a \emph{stable associative operad} $\mathsf{Ass}$ given by
\[ \mathsf{Ass}(I) = \Sigma^\infty_+ (\Sigma_I) \isom \Wdge_{\Sigma_I} S \]
where $\Sigma_I$ is the symmetric group on the set $I$, and whose algebras are non-unital associative $S$-algebras, and a \emph{stable commutative operad} $\mathsf{Com}$ given by
\[ \mathsf{Com}(I) = S, \]
for all $I$, whose algebras are non-unital commutative $S$-algebras.

(2) In \cite{ching:2005a}, we show that there is an operad $\der_*I$ in spectra whose \ord{n} term is equivalent ($\Sigma_n$-equivariantly) to the \ord{n} Goodwillie derivative of the identity functor $I$ on based spaces. The associated homology operad $H_*(\der_*I)$ is closely related to the Lie operad for graded vector spaces. The main motivation for this paper comes from joint work between the author and Greg Arone \cite{arone/ching:2009} on the relevance of the bar construction for operads in Goodwillie's calculus of functors.

(3) Let $G$ be a topological group. For an operad $\mathcal{P}$ in the symmetric monoidal category of unbased $G$-spaces, Westerland \cite{westerland:2006} constructs interesting operads $\mathcal{P}^{hG}$ and $\mathcal{P}_{bG}$, in spectra, that he calls the \emph{homotopy fixed point operad} and \emph{transfer operad}, respectively. In particular, if $\mathcal{P}$ is the little $2$-discs operad $\mathcal{D}_2$, and $G = S^1$, then the transfer operad $(\mathcal{D}_2)_{bS^1}$ is a spectrum-level version of the `gravity' operad considered by Getzler \cite{getzler:1994} and Ginzburg-Kapranov \cite{ginzburg/kapranov:1994}. (See also Kimura-Stasheff-Voronov \cite[2.5]{kimura/stasheff/voronov:1995}.)
\end{examples}

We now turn to the development of the various constructions for operads that this paper is about. We use certain categories of labelled rooted trees to make these constructions, so we recall those, following much of the terminology of \cite[\S3]{ching:2005a}.

\begin{definition}[Trees] \label{def:trees}
Let $I$ be a nonempty finite set. An \emph{$I$-tree} $T$ is a finite directed non-planar tree with a single terminal vertex (the \emph{root}) and a bijection between $I$ and the set of initial vertices (the \emph{leaves}). The root has exactly one incoming edge, the \emph{root edge}, and no outgoing edge. Each leaf has exactly one outgoing edge, a \emph{leaf edge}, and no incoming edges. The other vertices (the \emph{internal vertices}) have exactly one outgoing edge and at least two incoming edges. For smallness-sake, we restrict the vertices of our trees to lie in some fixed countable set such as $\mathbb{N}$. Thus there is only a set of $I$-trees for any given $I$.

An \emph{isomorphism of $I$-trees} is a bijection between directed graphs that preserves the labelling of the leaves. If such an isomorphism exists, it is unique. For $I$-trees $T,U$, we say that $T \leq U$ if $T$ is isomorphic to a tree obtained by contracting some set of internal edges in $U$. This relation determines a \emph{preorder} on the set of $I$-trees, that is, makes that set into a category in which each set of morphisms has at most one element. We denote this category by $\mathsf{T}(I)$.\footnote{We could construct this category more intrinsically by defining a morphism of $I$-trees to be a surjective function on the vertices of the trees that takes edges either to edges or to single vertices, and that preserves the labelling. If such a morphism exists between two $I$-trees then it is unique. In fact, this choice of morphisms produces the opposite of the category $\mathsf{T}(I)$. We reverse the direction to preserve the connection with the notation of \cite{ching:2005a} and so that $T < U$, in the preorder, when $T$ has fewer vertices than $U$.}

For each nonempty finite set $I$, there is an $I$-tree with no internal edges, that is unique up to isomorphism. We denote a choice of such tree by $\tau_I$. The tree $\tau_I$ is an initial object in $\mathsf{T}(I)$. If $I$ has only one element, then $\tau_I$ has no internal vertices and is the only element of $\mathsf{T}(I)$, up to isomorphism. For an edge $e$ in an $I$-tree $T$, we write $T/e$ for the $I$-tree obtained by contracting the edge $e$ and identifying its endpoints to a new vertex.

A bijection $\sigma: I \arrow{e,t}{\isom} I'$ determines an isomorphism of categories
\[ \sigma_*: \mathsf{T}(I) \arrow{e,t}{\isom} \mathsf{T}(I') \]
where $\sigma_*T$ has the same underlying tree as $T$, with a leaf labelled by $i$ in $T$, instead labelled by $\sigma(i)$ in $\sigma_*T$. We refer to $\sigma_*$ as the \emph{relabelling functor associated to $\sigma$}. In particular a permutation of $I$ determines a relabelling functor from $\mathsf{T}(I)$ to itself.

The following definitions play a particularly important role in this paper. First, if $t$ is an internal vertex of the $I$-tree $T$, we write $I_t$ for the set of incoming edges of $t$. If $U$ is another $I$-tree with $U \geq T$, then we can partition the vertices of $U$ according to the vertices of $T$ that they are identified with when collapsing edges in $U$ to form $T$. For the internal vertex $t$ of $T$, we write $U_t$ for the fragment of the tree $U$ formed by those vertices that collapse to $t$. The tree $U_t$ is naturally labelled by the set $I_t$. The following diagram illustrates this notation.
\end{definition}

\begin{center}
  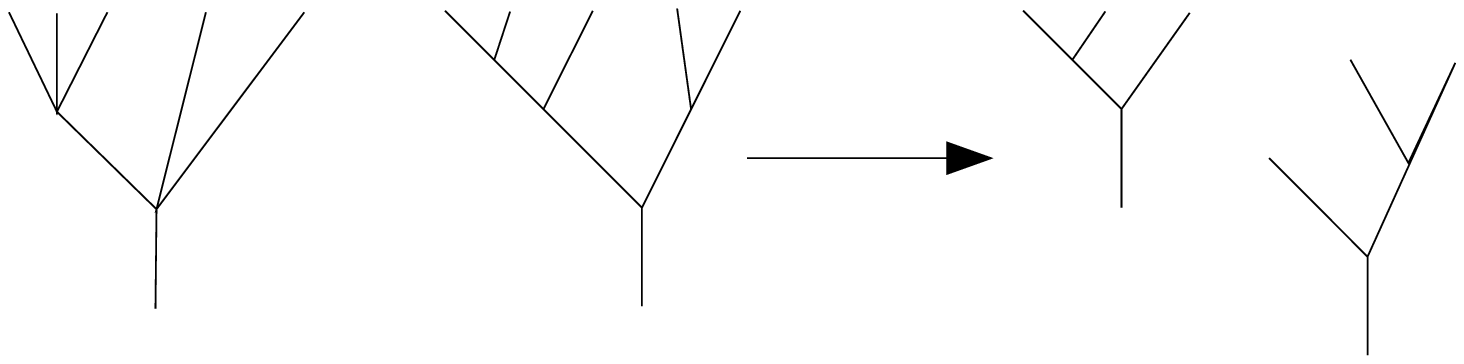
\end{center}

\begin{definition}[Grafting trees] \label{def:grafting}
Let $T$ be an $I$-tree, $U$ a $J$-tree and let $i \in I$. The \emph{grafted tree} $T \cup_i U$ is given by identifying the root edge of $U$ with the leaf edge of $T$ corresponding to the label $i$. The leaves of this tree are naturally labelled by the set $I \cup_i J := (I - \{i\}) \amalg J$.

\begin{center}
  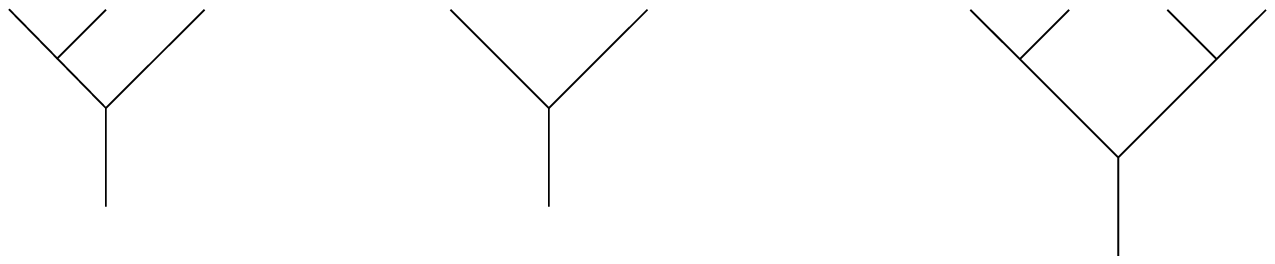
\end{center}

Note that if $T \leq U$ then the tree $U$ can be formed, up to isomorphism, by grafting together all of the $I_t$-trees $U_t$ for internal vertices $t$ in $T$.
\end{definition}

\begin{definition}[$A(T)$] \label{def:A(T)}
Let $A$ be a symmetric sequence. For an $I$-tree $T$, we define $A(T)$ to be the spectrum
\[ A(T) := \Smsh_{t \in T} A(I_t). \]
This smash product is indexed over the set of internal vertices of the tree $T$ and throughout this paper we write such indexing as over `$t \in T$'. Notice that we have natural isomorphisms
\[ A(\tau_I) \isom A(I) \]
and
\[ A(T \cup_i U) \isom A(T) \smsh A(U). \]
Also notice that $A(T)$ does not actually depend on the labelling of the leaves of $T$ by the elements of $I$. In particular, $A(T) = A(\sigma_*T)$ for any bijection $\sigma: I \arrow{e,t}{\isom} I'$.

An isomorphism $f$ between $I$-trees $T$ and $T'$ determines an isomorphism
\[ A(f): A(T) \arrow{e,t}{\isom} A(T') \]
as follows. Each vertex $t \in T$ corresponds under $f$ to a vertex $f(t) \in T'$. The isomorphism $f$ also determines a bijection between the set $I_t$ of incoming edges of $t$ in $T$, and the set $I_{f(t)}$ of incoming edges of $f(t)$ in $T'$. For each $t$, we therefore get an isomorphism
\[ A(I_t) \arrow{e,t}{\isom} A(I_{f(t)}). \]
Smashing these together over $t \in T$, we get the required isomorphism $A(f)$.
\end{definition}

\begin{lemma} \label{lem:operad-functor}
Let $P$ be a reduced operad. The assignment $T \mapsto P(T)$ determines a functor
\[ P(-): \mathsf{T}(I)^{op} \to \spectra, \]
for each nonempty finite set $I$, in such a way that the isomorphisms
\[ P(T \union_i U) \isom P(T) \smsh P(U) \]
are natural in $T$ and $U$. \qed
\end{lemma}
\begin{proof}
If $T/e$ is the $I$-tree obtained from $T$ by collapsing the internal edge $e$, identifying its endpoints $u,v$ to a new vertex $u \circ v$, the operad composition determines a map
\[ P(I_u) \smsh P(I_v) \smsh P(1) \smsh \dots \smsh P(1) \to P(I_{u \circ v}) \]
which, since $P$ is reduced, determines a map
\[ P(T) \to P(T/e). \]
If $f: T \to T'$ is an isomorphism of $I$-trees, we have an isomorphism
\[ P(f)^{-1}: P(T') \to P(T) \]
as in Definition \ref{def:A(T)}. Since the category $\mathsf{T}(I)$ is generated by morphisms of the form $T/e \to T$, together with the isomorphisms, these choices are enough to make $P(-)$ into a functor as claimed.
\end{proof}

\begin{definition}[Cooperads] \label{def:cooperad}
A \emph{cooperad} is a symmetric sequence $Q$ together with \emph{decomposition maps}
\[ Q(f): Q(I) \to Q(J) \smsh \Smsh_{j \in J} Q(I_j) \]
for each finite disjoint union $I = \coprod_{j \in J} I_j$, and a counit map $Q(1) \to S$, satisfying coassociativity and counit axioms. The cooperad $Q$ is \emph{reduced} if the counit map is an isomorphism. For a reduced cooperad $Q$, the decomposition maps make the assignments $T \mapsto Q(T)$ into functors
\[ Q(-): \mathsf{T}(I) \to \spectra, \]
for each nonempty finite set $I$, in such a way that the isomorphisms
\[ Q(T \union_i U) \isom Q(T) \smsh Q(U) \]
are natural in $T$ and $U$. A morphism of cooperads is a map of symmetric sequences that commutes with the structure maps. We have a category $\mathsf{Cooperad}$ of reduced cooperads of spectra.
\end{definition}

We now bring in the simplicial enrichment of the category $\spectra$ to define the Boardman-Vogt $W$-construction for operads of spectra (originally from \cite{boardman/vogt:1973}). Berger and Moerdijk \cite{berger/moerdijk:2006} have a general treatment of this construction in a symmetric monoidal category, of which our version is a special case.

\begin{remark} \label{rem:axiom}
Several of the constructions in this paper depend on the following property of the simplicial tensoring in $\spectra$. Let $X,Y \in \spectra$ and $K,L \in \sset$. Then there is a natural isomorphism
\[ d: (K \smsh L) \smsh (X \smsh Y) \arrow{e,t}{\isom} (K \smsh X) \smsh (L \smsh Y) \]
that satisfies appropriate unit and associativity properties (see \cite[1.10]{ching:2005a}).

To define the Boardman-Vogt $W$-construction, we actually only need a map in the backward direction of this isomorphism to exist. However, for the operadic bar construction in Definition \ref{def:bar}, we need a map in the forward direction to exist. To use and compare both the bar and $W$-constructions, as we do in this paper, we need both these maps to exist and be inverse isomorphisms.

%
\end{remark}

%
%
%
%

\begin{definition}[$W$-construction] \label{def:W-construction}
For an $I$-tree $T$, we write
\[ \Delta[T] := \prod_{e \in \mathsf{edge}(T)} \Delta[1] \]
for the product of copies of the standard simplicial interval $\Delta[1]$, indexed by the internal (i.e. non-root and non-leaf) edges of $T$. If $T \leq T'$ in $\mathsf{T}(I)$, there is a map of simplicial sets
\[ \iota_{T,T'}: \Delta[T] \to \Delta[T'] \]
given by setting the `new' edges in $T'$ (i.e. those that one collapses to form $T$) to have value $0$. Thus $\iota_{T,T'}$ is the inclusion of the cube $\Delta[T]$ as a face of $\Delta[T']$. Adding a disjoint basepoint, we obtain a functor $\Delta[-]_+: \mathsf{T}(I) \to \sset$ for each nonempty finite set $I$. We also have natural maps
\[ \mu_{T,i,U}: \Delta[T]_+ \smsh \Delta[U]_+ \to \Delta[T \union_i U]_+. \]
These are again given by the inclusion of one cube as a face of another. This time the `new' edge (that is, the edge of $T \union_i U$ where $T$ and $U$ are joined) is given value $1$. Notice that $\Delta[T]$ does not depend on the labelling of the leaves of $T$ by the elements of $I$. In particular, $\Delta[T] = \Delta[\sigma_*T]$ for any bijection $\sigma: I \arrow{e,t}{\isom } I'$.

Now let $P$ be a reduced operad. We define a new reduced operad $WP$ by
\[ WP(I) := \Delta[T]_+ \smsh_{T \in \mathsf{T}(I)} P(T). \]
This is a coend calculated over the category $\mathsf{T}(I)$ of $I$-trees. A bijection $\sigma: I \arrow{e,t}{\isom} I'$ determines a map $WP(\sigma): WP(I) \arrow{e,t}{\isom} WP(I')$ by identifying the term
\[ \Delta[T]_+ \smsh P(T) \]
in the coend for $WP(I)$ with the term
\[ \Delta[\sigma_*T]_+ \smsh P(\sigma_*T) \]
in the coend for $WP(I')$, by way of the identity map between these equal objects.

The operad structure on $WP$ is given by combining the maps $\mu_{T,i,U}$ with the isomorphisms
\[ P(T) \smsh P(U) \arrow{e,t}{\isom} P(T \union_i U). \]
Note that we use the inverse of the map $d$ of \ref{rem:axiom} to form the operad structure.
\end{definition}

\begin{remark}
Our definition of $WP$ is isomorphic to that denoted $W(H,P)$ by Berger and Moerdijk in \cite{berger/moerdijk:2006}, where $H$ is the `spectral interval'
\[ H := \Sigma^\infty \Delta[1]_+. \]
We often informally think of a `point' in $WP(I)$ as an $I$-tree together with lengths between $0$ and $1$ for its internal edges, and a decoration for each internal vertex $t$ from the object $P(I_t)$. Such a tree $T$ with internal edge $e$ of length zero is identified with the tree $T/e$ with new vertex decorated using the composition map for the operad $P$.
\end{remark}

\begin{definition}[Resolution by the $W$-construction] \label{def:W-map}
For an $I$-tree $T$ and reduced operad $P$, we have maps
\[ \eta_T: \Delta[T]_+ \smsh P(T) \to P(I) \]
given by the collapse map $\Delta[T] \to *$ and the operad composition map $P(T) \to P(I)$. These are natural in $T$ in the appropriate way and so together determine
\[ \eta_I: WP(I) \to P(I). \]
The $\eta_T$ also respect the grafting maps $\mu_{T,i,U}$ and we get a map of operads
\[ \eta: WP \to P. \]
\end{definition}

\begin{lemma} \label{lem:WP-P}
Let $P$ be a reduced operad. Then the map $\eta: WP \to P$ of Definition \ref{def:W-map} is a weak equivalence of operads in the projective model structure.
\end{lemma}
\begin{proof}
For an $I$-tree $I$, define a map of symmetric sequences $\zeta: P \to WP$ by
\[ \zeta_I: P(I) \arrow{e,t}{\isom} \Delta[\tau_I]_+ \smsh P(\tau_I) \to WP(I). \]
Note that $\zeta$ is not in general a map of operads.

The composite $\eta \zeta: P \to P$ is the identity. There is a simplicial homotopy $h_T$ between the identity map on $\Delta[T]$ and the constant map $\Delta[T] \to \Delta[T]$ to the point $(0,\dots,0)$ of the cube $\Delta[T]$. The homotopies $h_T$ can be chosen naturally in $T$ and determine a homotopy between the identity on $WP$ and the composite $\zeta \eta: WP \to WP$. It follows that $\zeta$ is a simplicial homotopy inverse to $\eta$ in the category of symmetric sequences. In particular, $\eta$ is a weak equivalence in the projective model structure.
\end{proof}

\section{The cobar-bar construction for an operad} \label{sec:bar}

We now consider the bar and cobar constructions for operads and cooperads of spectra. These were defined in \cite{ching:2005a} and are topological versions of the algebraic constructions of Getzler and Jones \cite{getzler/jones:1994}. Up to duality, they correspond to the `dg-dual' construction of Ginzburg and Kapranov \cite[\S3]{ginzburg/kapranov:1994}. The main aim of this section is show that the cobar-bar construction $CBP$, for an operad $P$, is weakly equivalent, \emph{in the category of operads}, to the $W$-construction $WP$. This improves on a result of the author and Greg Arone \cite[20.3]{arone/ching:2009} where we constructed such an equivalence at the level of symmetric sequences.

In this paper we use a slightly different description of $BP$ than that in \cite{ching:2005a}. This version is due to Salvatore \cite{salvatore:1998} and the two constructions yield isomorphic cooperads.

%
%
%
%

\begin{definition}[Bar construction] \label{def:bar}
For an $I$-tree $T$ we define
\[ w(T) := \begin{cases} \Delta[T] \times \Delta[1] & \text{if $|I| \geq 2$}; \\ \quad \quad * & \text{if $|I| = 1$}. \end{cases} \]
Let $w_0(T)$ be the sub-simplicial set consisting of the faces where, either, any edge has value $1$, or, the root edge has value $0$. We define
\[ \bar{w}(T) := w(T)/w_0(T) \]
and think of this as a pointed simplicial set with basepoint given by the quotient point. If $|I| = 1$, then $\bar{w}(T) = S^0$.

If $T \leq T'$, there is an inclusion map
\[ \iota_{T,T'}: \bar{w}(T) \to \bar{w}(T') \]
that identifies $\bar{w}(T)$ with the sub-simplicial set of $\bar{w}(T')$ in which the `new' edges of $T'$ have value $0$. These maps make $\bar{w}(-)$ into a functor $\mathsf{T}(I) \to \sset$. There are natural isomorphisms
\[ w(T \union_i U) \isom w(T) \times w(U) \]
in which the length for the `grafted' edge in $T \cup_i U$ is assigned to the root edge of $U$ and the length for the root edge in $T \cup_i U$ is assigned to the root edge of $T$. These pass to the respective quotients and give us
\[ \nu_{T,i,U}: \bar{w}(T \union_i U) \to \bar{w}(T) \smsh \bar{w}(U). \]
These maps are well-defined, natural in $T$ and $U$, and appropriately associative. Finally, notice that the simplicial set $\bar{w}(T)$ does not depend on the labelling of leaves of $T$ by elements of $I$. In particular, $\bar{w}(T) = \bar{w}(\sigma_*T)$ for any bijection $\sigma: I \arrow{e,t}{\isom} I'$.

Now let $P$ be a reduced operad. We define $BP$ to be the reduced symmetric sequence given by the coends
\[ BP(I) := \bar{w}(T) \smsh_{T \in \mathsf{T}(I)} P(T). \]
A bijection $\sigma: I \arrow{e,t}{\isom} I'$ determines a map $BP(I) \arrow{e,t}{\isom} BP(I')$ by identifying the term
\[ \bar{w}(T) \smsh P(T) \]
in the coend for $BP(I)$ with the term
\[ \bar{w}(\sigma_*T) \smsh P(\sigma_*T) \]
in the coend for $BP(I')$ via the identity map between these equal objects.

We give $BP$ a reduced cooperad structure by combining the maps $\nu_{T,i,U}$ above with the isomorphisms
\[ P(T \union_i U) \arrow{e,t}{\isom} P(T) \smsh P(U). \]
Here we need the maps $d$ of \ref{rem:axiom}.
\end{definition}

\begin{remark} \label{rem:bar-other}
We often think informally of a `point' in $BP(I)$ as an $I$-tree $T$ with lengths between $0$ and $1$ assigned to its internal edges, and to its root edge, together with a decoration for each internal vertex $t$ from the object $P(I_t)$. Trees in which the root edge has length $0$, or any edge has length $1$, are identified with the basepoint in $BP(I)$. A tree in which an internal edge has length $0$ is identified with the tree obtained by collapsing that edge and using the composition in the operad $P$ to decorate the new vertex.
\end{remark}

We now define the cobar construction for a reduced cooperad. To do this we employ the `reverses' of the simplicial sets $\bar{w}(T)$.

\begin{definition}[Reverse of a simplicial set] \label{def:reverse}
The simplicial indexing category $\Delta$ has an automorphism $R$ that sends a totally ordered finite set to its `opposite', that is, the same set with the opposite order. For a simplicial set $X$, the \emph{reverse} of $X$, denoted $X^{\mathsf{rev}}$ is the simplicial set $X \circ R$.
\end{definition}

\begin{definition}[Cobar construction for cooperads] \label{def:cobar}
Given a reduced cooperad $Q$, we define its \emph{cobar construction} $CQ$ to be the symmetric sequence
\[ CQ(I) := \Map_{T \in \mathsf{T}(I)}(\bar{w}(T)^{\mathsf{rev}},Q(T)). \]
This mapping spectrum is an `end' calculated over the category $\mathsf{T}(I)$. The notation $\Map(-,-)$ refers to the cotensoring of $\spectra$ over pointed simplicial sets. We make $CQ$ into a reduced operad by combining the reverses of the maps $\nu_{T,i,U}$ with the isomorphisms
\[ Q(T) \smsh Q(U) \arrow{e,t}{\isom} Q(T \union_i U). \]
Here we also require maps of the form
\[ d^*: \Map(K,X) \smsh \Map(L,Y) \to \Map(K \smsh L, X \smsh Y) \]
for simplicial sets $K,L$ and spectra $X,Y$. These can be constructed from adjoints of the maps labelled $d$ in Remark \ref{rem:axiom}.
\end{definition}

\begin{remark} \label{rem:reverse}
Since the tensoring of spectra over simplicial sets factors via geometric realization, it makes no actual difference to $CQ$ that we use $\bar{w}(T)^{\mathsf{rev}}$ instead of $\bar{w}(T)$. The simplicial sets $X$ and $X^{\mathsf{rev}}$ have homeomorphic realizations. However, we need to use the reversal to relate the cobar and bar constructions to the $W$-construction at the simplicial level.
\end{remark}

\begin{remark} \label{rem:cobar}
Informally, we think of a point in $CQ(I)$ as an assignment of a label $x \in Q(T)$ to each tree $T \in \mathsf{T}(I)$ whose internal and root edges have lengths between $0$ and $1$. If any edge in $T$ has length $1$, or if the root edge has length $0$, we assign the basepoint in $Q(T)$. If an internal edge $e$ in $T$ has length $0$, the assigned label in $Q(T)$ should be the image under the map $Q(T/e) \to Q(T)$ of the label assigned to the corresponding tree based $T/e$ with edge lengths the same as in $T$.
\end{remark}

\begin{definition}[Cobar-bar construction] \label{def:cobar-bar}
Let $P$ be a reduced operad of spectra. The \emph{cobar-bar construction on $P$} is the reduced operad $CBP$ formed by applying the cobar construction of Definition \ref{def:cobar} to the cooperad $BP$.
\end{definition}

We are now in a position to state the first main result of this paper: this is that, if $P$ is suitably cofibrant, $CBP$ is weakly equivalent to $P$ in the category of reduced operads. We prove this result by constructing a natural weak equivalence of operads
\[ \theta: WP \weq CBP. \]
This result should be compared to a theorem of Berger-Moerdijk \cite[8.5.4]{berger/moerdijk:2006} that, in the case of operads of chain complexes, $WP$ and $CBP$ are isomorphic. We describe the connection explicitly in Remark \ref{rem:Berger-Moerdijk} below.

\begin{informal-definition} \label{def:W-CB-informal}
Informally, we can think of the map $\theta$ in the following way. Start with a point $x \in WP(I)$, that is a tree $T$ whose internal edges have lengths between $0$ and $1$, together with a point $p \in P(T)$. To define the point $\theta(x) \in CBP(I)$, we have to assign, for each $I$-tree $U$ whose internal and root edges have lengths between $0$ and $1$, a point in $\theta(x)_U \in BP(U)$. Such a point in $BP(U)$ is, informally, a sequence of points, $\theta(x)_u \in BP(I_u)$ for each vertex $u$ of $U$, where $I_u$ is the set of incoming edges of $u$.

Firstly, if $U \nleq T$, we choose $\theta(x)_U$ to be the basepoint in $BP(U)$. If instead $U \leq T$, then each vertex $u$ of $U$ corresponds to a fragment $T_u$ of the tree $U$ in such a way that $T$ is obtained by grafting all the trees $T_u$ together. The required point $\theta(x)_u \in BP(I_u)$ is based on the tree $T_u$. Such a point requires a label from $P(T_u)$. This label is obtained from the original point $p \in P(T)$ which is itself equivalent to a sequence of points, one in $P(T_u)$ for each $u$. All that remains now is to specify the lengths of the internal and root edges of the trees $T_u$ that underlie the points $\theta(x)_u$. These depend on the corresponding lengths in the trees $T$ and $U$.

If $e$ is an internal edge of $T_u$, then it corresponds to a unique internal edge of the original tree $T$ and we give it the same length. If $e$ is the root edge of $T_u$, then it corresponds both to an edge in $T$ (say with length $t$), \emph{and} to an edge in $U$ (say with length $s$), namely the outgoing edge of the vertex $u$. In this case, we give $e$ length $\max(t-s,0)$. Finally, if $e$ is the root edge of $T_u$, where $u$ is the \emph{root vertex} of $U$, then $e$ corresponds to the root edge of $T$ which does not have a length. In this case, we act as though that root edge had length $1$ and give $e$ length $1-s$, where $s$ is the length of the root edge of $U$. The following picture illustrates an example of the map $\theta$ which hopefully clarifies the above description.

\begin{center}
  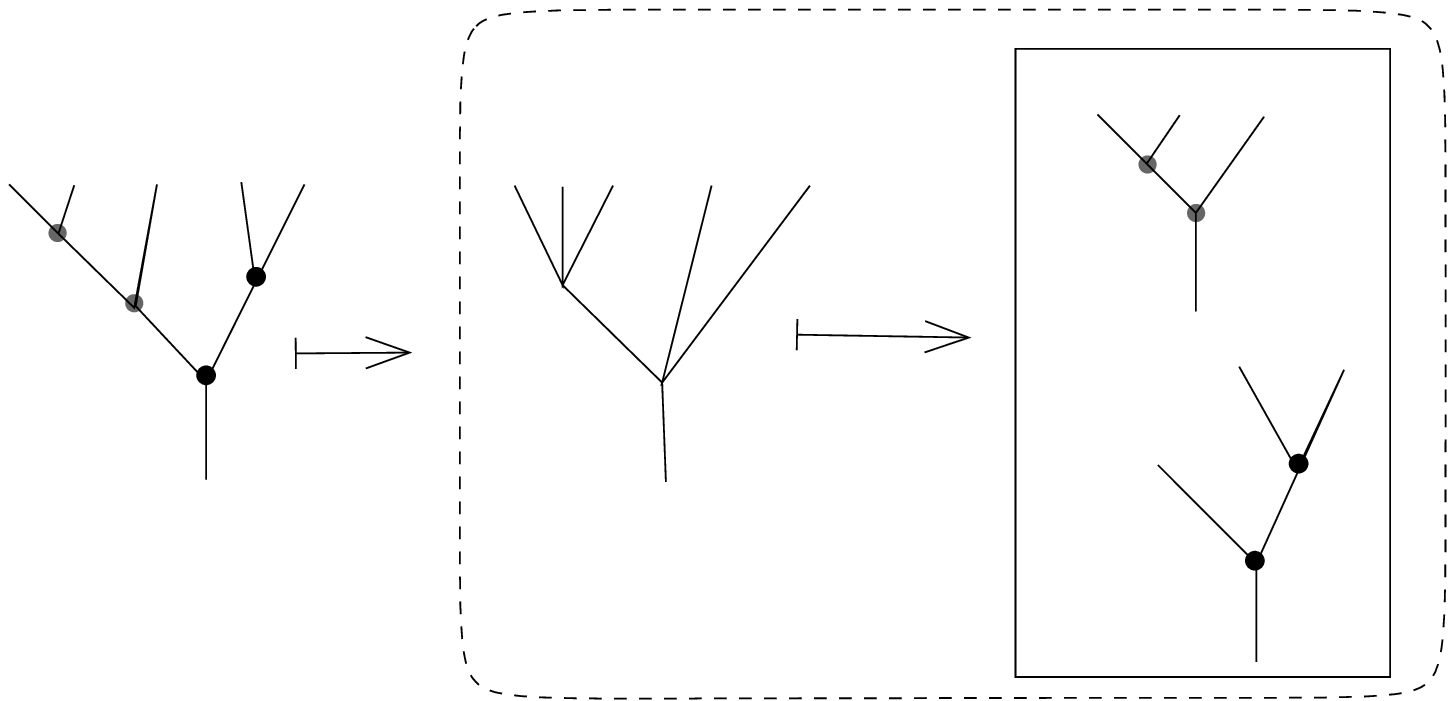
\end{center}

It remains to check that this definition really does give a well-defined operad map $\theta: WP \to CBP$. Instead of doing this now, we give a more formal definition of the map $\theta$ and verify that this is well-defined. Our definition relies on certain maps of simplicial sets which we now define.
\end{informal-definition}

\begin{definition}[The basic maps] \label{def:interval-map}
Define a map of simplicial sets
\[ h: \Delta[1] \times \Delta[1]^{\mathsf{rev}} \to \Delta[1] \]
whose realization is the map $(t,s) \mapsto \max(t-s,0)$ by the following picture:

\begin{center}
  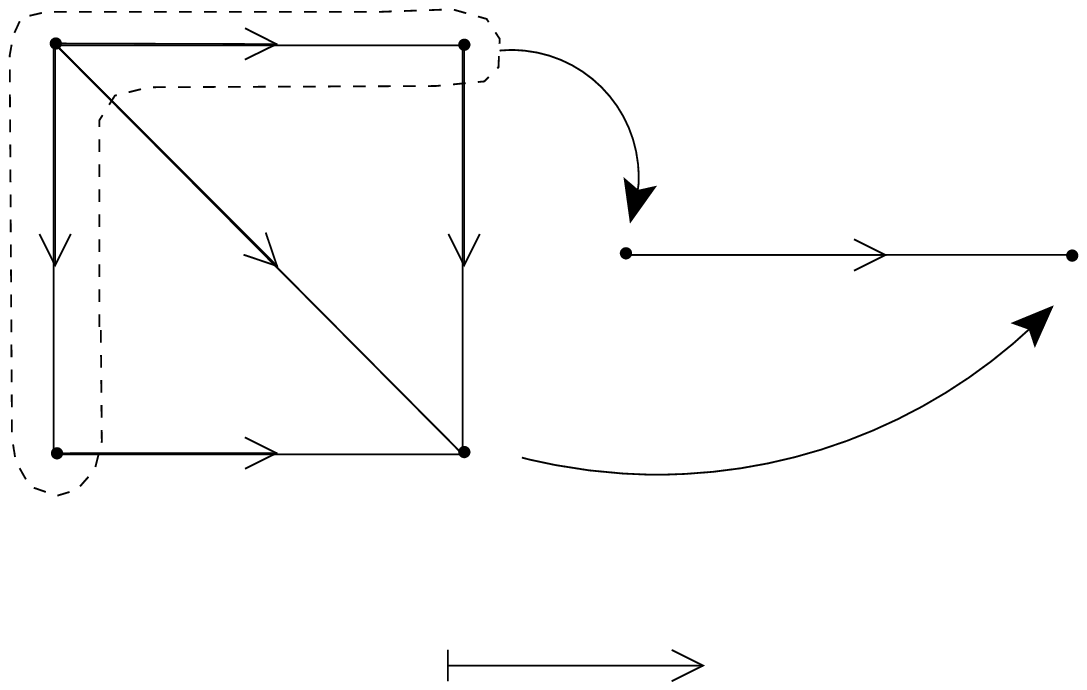
\end{center}

Also let $r$ denote the `reflection'
\[ r: \Delta[1]^{\mathsf{rev}} \to \Delta[1] \]
whose realization is the map
\[ s \mapsto 1-s. \]
It is to allow for the definition of the maps $r$ and $h$ that we have to be careful with the orientations of our intervals, using reversals in the definition of the cobar construction.
\end{definition}

\begin{definition}[Formal definition of $\theta$] \label{def:W-CB}
Fix a nonempty finite set $I$. Suppose $T,U \in \mathsf{T}(I)$ with $U \leq T$ and recall that we write $T_u$ for the fragment of $T$ that collapses to the vertex $u$ of $U$. The map $\theta$ is, at its heart, based on maps of simplicial sets of the form
\[ \theta_{T,U}: \Delta[T]_+ \smsh \bar{w}(U)^{\mathsf{rev}} \to \Smsh_{u \in U} \bar{w}(T_u). \]
This is a quotient of a map of cubes
\[ \hat{\theta}_{T,U}: \Delta[T] \times w(U)^{\mathsf{rev}} \to \prod_{u \in U} w(T_u). \]
The target of $\hat{\theta}_{T,U}$ is a product of copies of $\Delta[1]$ indexed by all the internal and root edges of the trees $T_u$ for all internal vertices $u \in U$. We define $\hat{\theta}_{T,U}$ component by component:
\begin{itemize}
 \item an internal edge $e$ of $T_u$ corresponds to a unique internal edge $e$ of $T$ and we choose the component of $\hat{\theta}_{T,U}$ corresponding to $e$ to be the projection onto the corresponding copy of $\Delta[1]$ in $\Delta[T]$;
 \item if $u$ is not the root vertex of $U$, then the root edge $e$ of $T_u$ corresponds to an internal edge of $T$ \emph{and} an internal edge of $U$ (the outgoing edge from $u$). In this case, we obtain the relevant component of $\hat{\theta}_{T,U}$ by projecting onto the copy of $\Delta[1] \times \Delta[1]^{\mathsf{rev}}$ in $\Delta[T] \times w(U)^{\mathsf{rev}}$ determined by these edges, and by applying the map $h$ of Definition \ref{def:interval-map};
 \item if $u$ is the root vertex of $U$, then the root edge of $T_u$ corresponds to the root edge of $U$, and we obtain the relevant component of $\hat{\theta}_{T,U}$ by projecting onto the copy of $\Delta[1]^{\mathsf{rev}}$ in $w(U)^{\mathsf{rev}}$ corresponding to this edge, and then applying the reflection map $r$ of Definition \ref{def:interval-map}.
\end{itemize}
Now let $P$ be a reduced operad. Using $\theta_{T,U}$ and the isomorphism $P(T) \arrow{e,t}{\isom} \Smsh_{U \in U} P(T_u)$ we obtain a map
\[ \Delta[T]_+ \smsh P(T) \to \Map\left(\bar{w}(U)^{\mathsf{rev}},\Smsh_{u \in U} \bar{w}(T_u) \smsh P(T_u)\right). \]
These are natural in $T,U \in \mathsf{T}(I)$ (see the second step in the proof of \ref{prop:W-CB} below) and so we obtain a single map
\[ \Delta[T]_+ \smsh_{T} P(T) \to \Map_{U}\left(\bar{w}(U)^{\mathsf{rev}},\Smsh_{u \in U} \bar{w}(T_u) \smsh_{T_u \in \mathsf{T}(I_u)} P(T_u)\right) \]
which is the required
\[ \theta(I): WP(I) \to CBP(I). \]
These respect the relabelling on trees so form a map of symmetric sequences
\[ \theta: WP \to CBP. \]
\end{definition}

\begin{prop} \label{prop:W-CB}
The construction of Definition \ref{def:W-CB} produces a well-defined morphism of operads $\theta: WP \to CBP$, natural in $P$.
\end{prop}
\begin{proof}
The first step is to check that $\hat{\theta}_{T,U}$ does pass to the quotient and defines $\theta_{T,U}$ as claimed. To explain this, it is simpler to use `topological' language by referring to `points' in the cubes $\Delta[T]$ and $w(U)^{\mathsf{rev}}$ as though we had taken geometric realization. It is also useful to think of these points as determining `lengths' for the edges of the trees $T$ and $U$.

If the outgoing edge of the vertex $u$ in $U$ has length $1$, then the corresponding root edge in $T_u$ is given length $0$, so determines the basepoint in $\bar{w}(T_u)$. If the root edge of $U$ has length $0$, then the root edge of $T_r$ (where $r$ is the root vertex of $U$) has length $1$, so again determines the basepoint. Thus the map $\theta_{T,U}$ is well-defined.

The second step is to check that the maps $\theta_{T,U}$ are natural in $T$ and $U$. This amounts to noticing three things:
\begin{itemize}
  \item if one of the edges in $T$ not in $U$ has length $0$, then the corresponding edge in the relevant $T_u$ also has length $0$;
  \item if one of the edges in $T$ that is in $U$ has length $0$, then the corresponding root edge in $T_u$ has length $0$ and so determines the basepoint;
  \item if the edge $e$ in $U$ has length $0$, then the corresponding root edge in $T_u$ has the same length as its length in $T$. This is the same as the length it would have had as an internal edge in the larger fragment $T_u$ that we would have obtained if $U$ were replaced with the smaller tree $U/e$.
\end{itemize}

The third and final step is to consider what happens if $T = T_1 \cup_i T_2$ (with the `new' internal edge of length $1$) and $U = U_1 \cup_i U_2$ with $U_j \leq T_j$ for $j = 1,2$. It is easy to check that the map $\hat{\theta}_{T,U}$ is then essentially the product of the maps $\hat{\theta}_{T_1,U_1}$ and $\hat{\theta}_{T_2,U_2}$. This ensures that $\theta$ is a map of operads as required.
\end{proof}

\begin{remark} \label{rem:Berger-Moerdijk}
We describe the connection between our map $\theta: WP \to CBP$ and Theorem 8.5.4 of \cite{berger/moerdijk:2006}. The latter states that there is an isomorphism of operads
\[ W(H,P) \isom CBP \]
where $P$ is a reduced operad of chain complexes of $R$-modules, $H$ is the following `interval' in the category of chain complexes
\[ 0 \leftarrow R\{\gamma_0\} \oplus R\{\gamma_1\} \leftarrow R\{\gamma\} \leftarrow 0 \leftarrow \dots, \]
and $W(H,P)$ is the $W$-construction based on the interval $H$, as defined in \cite[\S4]{berger/moerdijk:2006}. Here $C$ and $B$ denote, respectively, the cobar and bar constructions for cooperads and operads of chain complexes, as described by Getzler and Jones \cite{getzler/jones:1994}.

We leave the reader to check that, using the interval $H$ in place of $\Delta[1]$, the cobar and bar constructions of \S2 above yield precisely the algebraic constructions of Getzler and Jones. (This was essentially done in \cite[9.4]{ching:2005a} though from a slightly different perspective.) Similarly, using $H$ for the $W$-construction of \S1 yields precisely the Berger-Moerdijk version. The construction of the map $\theta$ in Definition \ref{def:W-CB} then carries over to the algebraic setting to determine a map
\[ \theta: W(H,P) \to CBP. \]
This construction involves maps of chain complexes that correspond to the maps $h$ and $r$ of Definition \ref{def:interval-map}. Note that in the algebraic case $H$ is its own `reverse', that is, there is an isomorphism of chain complexes $r: H \to H$ that sends $\gamma_0$ to $\gamma_1$ and vice versa.

Finally, we can see that the map $\theta$ is an isomorphism in this case by comparing its construction to the isomorphism described by Berger and Moerdijk in \cite[8.5.4]{berger/moerdijk:2006}.
\end{remark}

Returning to the case of spectra, our next goal is to show that $\theta$ is a weak equivalence of operads in the projective model structure (when $P$ is suitably cofibrant). We first describe the cofibrancy condition required.

\begin{definition}[Termwise-cofibrant operads and cooperads]
Let $A$ be a reduced symmetric sequence, operad or cooperad. We say that $A$ is \emph{termwise-cofibrant} if, for each nonempty finite set $I$ with $|I| \geq 2$, the object $A(I)$ is cofibrant in the standard model structure on $\spectra$. Note that the object $A(1)$ is isomorphic to the sphere spectrum $S$, hence not cofibrant in the EKMM model structure.

In \cite[\S9]{arone/ching:2009} we proved, with Greg Arone, that the category of reduced operads has termwise-cofibrant replacements, given by actual cofibrant replacements in the projective model structure. Thus, given an operad $P$, there is a natural weak equivalence of operads
\[ \tilde{P} \weq P \]
where $\tilde{P}$ is termwise-cofibrant. For the rest of this paper, we use this notation (that is, adding a tilde) to denote such a termwise-cofibrant replacement. For example, $\tilde{C}Q$ denotes a termwise-cofibrant replacement of the operad $CQ$.
\end{definition}

\begin{prop} \label{prop:bar-cobar-homotopy}
Let $f: P \weq P'$ be a weak equivalence of termwise-cofibrant operads. Then the induced map
\[ Bf: BP \to BP' \]
is a weak equivalence of termwise-cofibrant cooperads. Dually, let $g:Q \weq Q'$ be a weak equivalence of termwise-cofibrant cooperads. Then the induced map
\[ Cg: CQ \to CQ' \]
is a weak equivalence of operads.
\end{prop}
\begin{proof}
The first statement is \cite[8.5]{arone/ching:2009}, but we give a slightly different proof that dualizes to the second statement. For each nonempty finite set $I$, the category $\mathsf{T}(I)$ is Reedy (see \cite[\S15]{hirschhorn:2003}) with degree function given by the number of internal vertices in a tree. There is therefore a Reedy model structure on the categories of functors $\mathsf{T}(I) \to \spectra$ and $\mathsf{T}(I)^{op} \to \spectra$.

In the Reedy category $\mathsf{T}(I)^{op}$ every non-identity morphism decreases degree so Reedy cofibrant diagrams are just the objectwise cofibrant diagrams. In particular, the diagram $P(-): \mathsf{T}(I)^{op} \to \spectra$ of Lemma \ref{lem:operad-functor} is Reedy cofibrant when $P$ is a termwise-cofibrant operad. Thus $f$ determines an objectwise weak equivalence of Reedy cofibrant diagrams $P(T) \weq P'(T)$.

To see that $Bf$ is a weak equivalence, it is sufficient, by \cite[18.4.13]{hirschhorn:2003}, to show that the functor
\[ \bar{w}(-): \mathsf{T}(I) \to \sset \]
is Reedy cofibrant, for all nonempty finite sets $I$. For a given $I$-tree $T$, the latching object
\[ \colim_{T' < T} \bar{w}(T') \]
is the sub-simplicial set of $\bar{w}(T)$ given by those faces of $w(T)$ the correspond to some internal edge having length $0$. To see this, note that for $T' < T'' \leq T$, the map $\bar{w}(T') \to \bar{w}(T'')$ is an inclusion of simplicial sets. It follows that the latching map
\[ \colim_{T' < T} \bar{w}(T') \to \bar{w}(T) \]
is a cofibration of simplicial sets, hence $\bar{w}(-)$ is a Reedy cofibrant diagram, as required.

For the second statement, notice similarly that the Reedy fibrant diagrams $\mathsf{T}(I) \to \spectra$ are the objectwise fibrant diagrams, that is, all the diagrams (since every object in $\spectra$ is fibrant). Hence $g$ induces a weak equivalence $Q(-) \weq Q'(-)$ of Reedy fibrant diagrams. (Here we use the condition that $Q$ and $Q'$ are termwise-cofibrant to ensure that $Q(T) \to Q'(T)$ is a weak equivalence.) The reversed simplicial sets $\bar{w}(T)^{\mathsf{rev}}$ still form a Reedy cofibrant diagram and so, again by \cite[18.4.13]{hirschhorn:2003}, $Cg$ is a weak equivalence.
\end{proof}

\begin{theorem} \label{thm:W-CB}
Let $P$ be a termwise-cofibrant reduced operad in $\spectra$. Then the map
\[ \theta: WP \to CBP \]
of Definition \ref{def:W-CB} is a weak equivalence of operads.
\end{theorem}

A significant piece of the proof of this was given in \cite[20.3]{arone/ching:2009} where, together with Greg Arone, we proved most of the following result.

\begin{proposition} \label{prop:P-CBP-symseq}
For a termwise-cofibrant reduced operad $P$, there is a natural zigzag of weak equivalences of \emph{symmetric sequences}
\[ P \homeq CBP. \]
\end{proposition}

Note that this zigzag includes non-operad maps, so does not immediately imply Theorem \ref{thm:W-CB}. We previously proved this result for the double Koszul dual $KKP$ instead of $CBP$, though these are equivalent under finiteness hypotheses. In section \ref{sec:P-CBP} below we prove the full version of Proposition \ref{prop:P-CBP-symseq}.

In this section we show how Theorem \ref{thm:W-CB} follows from \ref{prop:P-CBP-symseq}. We do this by considering the tower of `truncations' of the operad $P$.

\begin{definition}[Truncated operads] \label{def:truncation}
Let $P$ be a termwise-cofibrant reduced operad. For an integer $n \geq 1$, the \emph{\ord{n} truncation} of $P$ is the reduced operad $P_{\leq n}$ given by
\[ P_{\leq n}(I) := \begin{cases} P(I) & \text{if $|I| \leq n$}; \\ \; \; * & \text{otherwise}. \end{cases} \]
with composition and unit maps equal to those for $P$ except when those maps are forced to be trivial.

We also define the \emph{\ord{n} layer} of $P$ to be the reduced operad $P_{=n}$ given by
\[ P_{=n}(I) := \begin{cases} P(I) & \text{if $|I| = n$}; \\ \; \; S & \text{if $|I| = 1$}; \\  \; \; * & \text{otherwise}. \end{cases} \]
with the trivial operad structure.
\end{definition}

Notice that for $n \geq 2$, there is a natural sequence of reduced operads
\[ P_{=n} \to P_{\leq n} \to P_{\leq(n-1)}. \]
This is a \emph{homotopy-fibre sequence} of reduced operads in the sense that
\[ P_{=n}(I) \to P_{\leq n}(I) \to P_{\leq(n-1)}(I) \]
is a homotopy-fibre sequence of spectra whenever $|I| \geq 2$.

Now consider the following diagram of spectra
\[ \tag{*} \begin{diagram}
  \node{WP_{=n}(I)} \arrow{e,t}{\theta} \arrow{s} \node{CBP_{=n}(I)} \arrow{s} \\
  \node{WP_{\leq n}(I)} \arrow{e,t}{\theta} \arrow{s} \node{CBP_{\leq n}(I)} \arrow{s} \\
  \node{WP_{\leq(n-1)}(I)} \arrow{e,t}{\theta} \node{CBP_{\leq(n-1)}(I)}
\end{diagram} \]
where $|I| \geq 2$. The left-hand column is a homotopy-fibre sequence of spectra by Lemma \ref{lem:WP-P}. The right-hand column is a homotopy-fibre sequence of spectra by Proposition \ref{prop:P-CBP-symseq}.

We now show that the top horizontal map is a weak equivalence for any $n \geq 1$ and any nonempty finite set $I$. From this we deduce, by induction on $n$, that $\theta$ is an equivalence for any $P_{\leq n}$.

In fact, we prove that $\theta: WA \to CBA$ is a weak equivalence for any termwise-cofibrant $A$ with a trivial operad structure. We first show by another means that $CBA$ is equivalent to $A$ when $A$ is trivial, and then show this equivalence is compatible with $\theta$.

\begin{definition} \label{def:omega-sigma}
For a reduced symmetric sequence $A$, we write $\Omega A$ for the reduced symmetric sequence with
\[ \Omega A(I) := \Map((S^1)^{\mathsf{rev}},A(I)) \]
and $\Sigma A$ for the reduced symmetric sequence
\[ \Sigma A(I) := S^1 \smsh A(I), \]
for finite sets $I$ with $|I| \geq 2$. In both cases $S^1$ is the simplicial circle $\Delta[1]/\partial\Delta[1]$.
\end{definition}

\begin{lemma} \label{lem:omega-sigma}
For a termwise-cofibrant trivial reduced operad $A$, there is a weak equivalence of operads
\[  \epsilon: CBA \weq \Omega \Sigma A \]
where the reduced symmetric sequence $\Omega \Sigma A$ is given the trivial operad structure.
\end{lemma}
\begin{proof}
Because the operad structure maps in $A$ are trivial, $BA(I)$ splits up as
\[ BA(I) \isom \Wdge_{[T]} (\bar{w}(T)/\bar{w}_1(T)) \smsh A(T) \isom \Wdge_{[T]} (\Sigma A)(T) \]
where $\bar{w}_1(T)$ is the subspace of $\bar{w}(T)$ given by the faces where one of the edges of the tree $T$ is assigned length $0$. Here the wedge product is taken over isomorphism classes of trees in $\mathsf{T}(I)$.

If $U$ is another $I$-tree, then we get from this
\[ BA(U) \isom \Wdge_{[T] \geq [U]} (\Sigma A)(T) \weq \prod_{[T] \geq [U]} (\Sigma A)(T) \]
where the wedge and product are over isomorphisms classes of $T$ such that $T \geq U$. Now we showed in the proof of Proposition \ref{prop:bar-cobar-homotopy} that $\bar{w}(-)^{\mathsf{rev}}$ is a Reedy cofibrant $\mathsf{T}(I)$-diagram of simplicial sets. It follows that we have a weak equivalence
\[ \epsilon_1: CBA(I) \weq \Map_{U \in \mathsf{T}(I)} \left( \bar{w}(U)^{\mathsf{rev}},\prod_{[T] \geq [U]} (\Sigma A)(T) \right). \]
The latter object is isomorphic to
\[ \prod_{[T]} \Map(\bar{w}(T)^{\mathsf{rev}}, (\Sigma A)(T)). \]
The only $T$ for which $\bar{w}(T)^{\mathsf{rev}}$ is not contractible is $T = \tau_I$. The projection from this product onto the term for $T = \tau_I$ is therefore a weak equivalence. The composite of this projection with $\epsilon_1$ is our weak equivalence $\epsilon$.

It remains to check that $\epsilon$ is a morphism of operads where $\Omega \Sigma A$ is given the trivial operad structure. From the definition of the operad structure on $CBA$, nontrivial products act trivially on the tree $\tau_I$, so under $\epsilon$ map trivially to $\Omega \Sigma A$.
\end{proof}

Now we can complete the proof of our main result.

\begin{proof}[Proof of Theorem \ref{thm:W-CB}]
First notice that if $|I| \leq n$, then
\[ WP(I) \isom WP_{\leq n}(I), \quad CBP(I) \isom CBP_{\leq n}(I) \]
so it is sufficient to prove that $\theta: WP_{\leq n} \to CBP_{\leq n}$ is a weak equivalence for all $n$. We do this by induction on $n$ using the diagram (*). This reduces to proving that $\theta: WA \to CBA$ is a weak equivalence for a trivial reduced operad $A$.

We claim that there is a commutative diagram of symmetric sequences
\[ \tag{**} \begin{diagram}
  \node{WA} \arrow{e,t}{\theta} \arrow{s,lr}{\eta}{\sim} \node{CBA} \arrow{s,lr}{\sim}{\epsilon} \\
  \node{A} \arrow{e,tb}{\sim}{r^{\#}} \node{\Omega \Sigma A}
\end{diagram} \]
The bottom horizontal map is adjoint to the `flip' map
\[ r \smsh A: (S^1)^{\mathsf{rev}} \smsh A \to S^1 \smsh A \]
determined by the reflection $r: (S^1)^{\mathsf{rev}} \to S^1$. This is a weak equivalence, for termwise-cofibrant $A$, since $\spectra$ is a stable simplicial model category.

To check that (**) is commutative, consider $WA$ for a trivial operad $A$. This splits up as
\[ WA(I) \isom \Wdge_{T \in \mathsf{T}(I)} \left[\Delta[T]/\Delta_0[T]\right] \smsh A(T) \]
where $\Delta_0[T]$ is the subspace of the cube $\Delta[T]$ consisting of the faces for which one of the internal edges of $T$ is assigned length $0$. The map $\eta$ is the collapse onto the factor for $T = \tau_I$.

Following through the definition of $\theta$ and $\epsilon$, we see that the composite
\[ WA(I) \to CB(I) \to \Omega \Sigma A(I) \]
is the trivial map on all the terms for which $T \neq \tau_I$. For the term $T = \tau_I$ it is the flip map $r^{\#}: A(I) \to \Omega \Sigma A(I)$. Thus the diagram commutes.

Since all the other maps in the diagram (**) are weak equivalences, it follows that $\theta$ is also a weak equivalence. This completes the proof of Theorem \ref{thm:W-CB}.
\end{proof}

We have now shown that a reduced operad $P$ can be recovered, up to weak equivalence, from its bar construction $BP$ together with the cooperad structure. In the next section we answer the dual question: can a reduced cooperad $Q$ be recovered from its cobar construction $CQ$? One might hope to do this by dualizing the approach of this section and construct weak equivalences of cooperads of the form
\[ BCQ \weq W^cQ \lweq Q \]
where $W^cQ$ is a `co-$W$-construction' for cooperads. However, the $W$-construction does not dualize immediately. This is because there is no dual version of the isomorphism $d$ of Remark \ref{rem:axiom}. In Definition \ref{def:cobar} we used a map $d^*$ that is dual to the map $d$, but $d^*$ is not in general an isomorphism. We would need an inverse to $d^*$ to form the `co-$W$-construction'.

In the next section, we solve this problem by expanding our notion of a cooperad slightly. This change allows for the existence of $W^cQ$ and of the weak equivalences relating it to $Q$ and $BCQ$.

\section{A model for the homotopy theory of cooperads} \label{sec:precooperads}

In this section we describe a model category $\mathsf{PreCooperad}$ that contains the category $\mathsf{Cooperad}$ of reduced cooperads as a full subcategory. Every `pre-cooperad' is weakly equivalent, in this model structure, to a termwise-cofibrant cooperad. We extend the cobar construction $C$ from cooperads to pre-cooperads and show that $C$ is the right adjoint of a Quillen equivalence between $\mathsf{Operad}$ (with the projective model structure) and $\mathsf{PreCooperad}$. The left adjoint is not precisely the bar construction, but is equivalent to it (at least on cofibrant operads).

\begin{remark}
Our results apply to operads in other models for the stable homotopy category. For example, Kro \cite{kro:2007} has shown that the category of operads in orthogonal spectra inherits a projective model structure (from the positive stable model structure on orthogonal spectra). There is a Quillen equivalence between this model category and a corresponding model structure on pre-cooperads in orthogonal spectra. Guti{\'e}rrez and Vogt \cite{gutierrez/vogt:2010} have done the same thing for symmetric spectra using work of Elmendorf and Mandell \cite{elmendorf/mandell:2006}.
\end{remark}

We start by describing the category $\mathsf{PreCooperad}$. For this we need to collect all the individual sets of trees $\mathsf{T}(I)$ into a single category, and add in morphisms that correspond to relabelling.

\begin{definition}[The category $\mathsf{Tree}$] \label{def:tree-category}
Let $\mathsf{Tree}$ denote the category whose objects are the $I$-trees for all finite sets $I$ with $|I| \geq 2$, and for which a morphism $T \to T'$, where $T \in \mathsf{T}(I)$ and $T' \in \mathsf{T}(I')$, consists of a bijection $\sigma: I \to I'$ such that $\sigma_*T \leq T'$ in $\mathsf{T}(I')$. Composition in $\mathsf{Tree}$ is composition of bijections -- this is well-defined because if $\sigma_*T \leq T'$ and $\rho_*T' \leq T''$ then $\rho_*\sigma_*T \leq T''$.
\end{definition}

\begin{definition}[Pre-cooperads] \label{def:precooperad}
A \emph{pre-cooperad} $Q$ consists of a functor
\[ Q(-): \mathsf{Tree} \to \spectra \]
and natural maps
\[ m_{T,i,U}: Q(T) \smsh Q(U) \to Q(T \cup_i U) \]
where $T \in \mathsf{T}(I)$, $U \in \mathsf{T}(J)$ and $i \in I$. The maps $m_{T,i,U}$ are required to be associative in an appropriate sense. A \emph{morphism of pre-cooperads} $Q \to Q'$ consists of natural transformations $Q(T) \to Q'(T)$ that commute appropriately with the maps $m_{T,i,U}$. We thus obtain a category $\mathsf{PreCooperad}$ of pre-cooperads and their morphisms.
\end{definition}

\begin{example} \label{ex:precooperad}
In Definition \ref{def:cooperad} we saw that any reduced cooperad $Q$ determines a pre-cooperad via Definition \ref{def:A(T)}. In this case the maps $m_{T,i,U}$ are all isomorphisms. Conversely, given a pre-cooperad $Q$ in which the maps $m_{T,i,U}$ are isomorphisms, we can define a reduced cooperad, which we also denote $Q$, by setting $Q(I) := Q(\tau_I)$. The decomposition maps are then given by the composites
\[ Q(I \cup_i J) = Q(\tau_{I \cup_i J}) \to Q(\tau_I \cup_i \tau_J) \isom Q(\tau_I) \smsh Q(\tau_J) = Q(I) \smsh Q(J). \]
This construction identifies the category of reduced cooperads with a subcategory of the category of pre-cooperads. If the pre-cooperad $Q$ is actually a cooperad, then any morphism of pre-cooperads $Q \to Q'$ is determined by its value on the terms $Q(\tau_I)$. This tells us that the reduced cooperads form a \emph{full} subcategory of $\mathsf{PreCooperad}$.
\end{example}

In some ways the key observation of this section is that Definition \ref{def:cobar} of the cobar construction $CQ$ does not require that $Q$ be a cooperad. It is sufficient for $Q$ to be a pre-cooperad.

\begin{definition}[Cobar construction for pre-cooperads] \label{def:cobar-precooperad}
For each pre-cooperad $Q$, we define $CQ$ to be the reduced operad given by
\[ CQ(I) := \Map_{T \in \mathsf{T}(I)}(\bar{w}(T),Q(T)) \]
with operad composition maps given by combining the $\nu_{T,i,U}$ of Definition \ref{def:bar} with the $m_{T,i,U}$ of Definition \ref{def:precooperad}. (As with the original definition of $CQ$, we are also using the maps $d^*$ of Remark \ref{rem:axiom}.) We thus obtain a functor
\[ C: \mathsf{PreCooperad} \to \mathsf{Operad}. \]
\end{definition}

Our first main goal is to show that $C$ is the right adjoint of a Quillen equivalence between operads and pre-cooperads. We start by describing a `strict' model structure on pre-cooperads, of which the model structure we are really interested in is a localization. The strict model structure has weak equivalences and fibrations defined termwise. To see that this indeed determines a model structure, we introduce `free' pre-cooperads.

\begin{definition}[Free pre-cooperads] \label{def:tree-monad}
We write $\spectra^{\mathsf{Tree}}$ for the category of functors $\mathsf{Tree} \to \spectra$. For $A \in \spectra^{\mathsf{Tree}}$, we define $\mathbb{F}A \in \spectra^{\mathsf{Tree}}$ by
\[ \mathbb{F}A(T) := \colim_{U \leq T} \Smsh_{u \in U} A(T_u) \isom \Wdge_{[U] : U \leq T} \Smsh_{u \in U} A(T_u). \]
This is the colimit calculated over the subcategory of $\mathsf{T}(I)$ consisting of the $I$-trees $U$ with $U \leq T$ and isomorphisms between them. Because there is at most a unique isomorphism between any two $I$-trees, this colimit is isomorphic to a coproduct taken over isomorphism classes of $I$-trees $U$ with $U \leq T$. The smash product is taken over all internal vertices $u \in U$ and $T_u$ refers to the part of the tree $T$ that collapses to the vertex $u$ under the collapse map determined by the inequality $U \leq T$.

Given a morphism $\sigma: T \to T'$ in $\mathsf{Tree}$ and $U \leq T$, we have $\sigma_*U \leq \sigma_*T \leq T'$. For $u \in U$, we can identify $T_u$ with the piece of $\sigma_*T$ that collapses to the corresponding vertex $u$ of $\sigma_*U$. Thus we have $T_u \leq T'_u$ and so the functor $A$ determines a map $A(T_u) \to A(T'_u)$. Putting such maps together, we get a map
\[ \mathbb{F}A(T) \to \mathbb{F}A(T') \]
that makes $\mathbb{F}A$ into a functor $\mathsf{Tree} \to \spectra$, and $\mathbb{F}$ into a functor $\spectra^{\mathsf{Tree}} \to \spectra^{\mathsf{Tree}}$.

We define a monad structure on $\mathbb{F}$ by noticing that
\[ \mathbb{F}\mathbb{F}A(T) \isom \colim_{V \leq U \leq T} \Smsh_{u \in U} A(T_u). \]
The composition map $\mathbb{F}\mathbb{F} \to \mathbb{F}$ is given by `forgetting' the variable $V$ in the indexing set for the wedge sum. The unit map $A \to \mathbb{F}A$ is given by the inclusions of $A(T)$ as the term corresponding to $U = \tau_I$.

These definitions make $\mathbb{F}A$ into a pre-cooperad for any $A \in \spectra^{\mathsf{Tree}}$ and we refer to $\mathbb{F}A$ as the \emph{free pre-cooperad on $A$}.
\end{definition}

\begin{lemma} \label{lem:precooperad}
The category $\mathsf{PreCooperad}$ is equivalent to the category of algebras over the monad $\mathbb{F}$. \qed
\end{lemma}

\begin{remark} \label{rem:free-precooperad}
Interestingly, the free pre-cooperads correspond in some sense to trivial cooperads. Let $A$ be a symmetric sequence and extend $A$ to a functor $\mathsf{Tree} \to \spectra$ by setting
\[ A(T) = \begin{cases}
  A(I) & \text{if $T = \tau_I$ for some $I$}; \\
  \; \; * & \text{otherwise}.
\end{cases} \]
Then we have
\[ \mathbb{F}A(T) \isom \Smsh_{t \in T} A(I_t). \]
The pre-cooperad structure maps
\[ \mathbb{F}A(T) \smsh \mathbb{F}A(U) \to \mathbb{F}A(T \cup_i U) \]
are isomorphisms, meaning that the pre-cooperad $\mathbb{F}A$ is in this case an actual cooperad. Furthermore, for any nontrivial $T \to T'$ (i.e. not just a relabelling) the induced map $\mathbb{F}A(T) \to \mathbb{F}A(T')$ is trivial so that $\mathbb{F}A$ is the trivial cooperad based on the symmetric sequence $A$.
\end{remark}

\begin{prop} \label{prop:strict-precooperad}
The category of pre-cooperads is enriched, tensored and cotensored over the category of pointed simplicial sets and there is a right proper cellular simplicial model structure on $\mathsf{PreCooperad}$ in which a map $Q \to Q'$ is a weak equivalence (or, respectively, a fibration) if and only if $Q(T) \to Q'(T)$ is a weak equivalence (respectively a fibration) in $\spectra$, for each $T \in \mathsf{Tree}$. We refer to this as the \emph{strict model structure} on the category of pre-cooperads. These weak equivalences are the \emph{strict weak equivalences} of pre-cooperads, and these fibrations are the \emph{strict fibrations}.
\end{prop}
\begin{proof}
The proof of this is virtually identical to the proof that the category of operads in $\spectra$ has a projective model structure. (See \cite[Appendix]{arone/ching:2009} which follows the approach of EKMM \cite[VII]{elmendorf/kriz/mandell/may:1997}.) Replace the category of reduced symmetric sequences $\spectra^{\mathsf{\Sigma}_+}$ with the category $\spectra^{\mathsf{Tree}}$, and replace the free operad functor $F$ with the free pre-cooperad functor $\mathbb{F}$.
\end{proof}

The following lemma is also useful.

\begin{lemma} \label{lem:strictly-cofibrant-precooperad}
Let $Q$ be a strictly-cofibrant pre-cooperad. Then $Q$ is \emph{termwise-cofibrant}, that is $Q(T)$ is a cofibrant spectrum for all $T \in \mathsf{Tree}$.
\end{lemma}
\begin{proof}
This is the analogue of \cite[9.14]{arone/ching:2009}. The proof relies of the analogue of the `Cofibration Hypothesis' for pre-cooperads. (See \cite[A.11]{arone/ching:2009}.)
\end{proof}

The model structure on pre-cooperads that we are really interested in is a right Bousfield localization of the strict model structure, in the sense of Hirschhorn \cite[3.3.1]{hirschhorn:2003}. Its weak equivalences are detected by the cobar construction.

\begin{definition}[$C$-equivalences] \label{def:weq-precooperad}
A morphism $Q \to Q'$ of pre-cooperads is a \emph{$C$-equivalence} if it induces a weak equivalence $CQ \to CQ'$ of operads.
\end{definition}

\begin{prop} \label{prop:model-precooperad}
The $C$-equivalences of Definition \ref{def:weq-precooperad} and strict fibrations of Proposition \ref{prop:strict-precooperad} determine a right proper cellular simplicial model structure on the category $\mathsf{PreCooperad}$. We refer to this as the \emph{$C$-model structure}.
\end{prop}
\begin{proof}
This is an application of the localization machinery of Hirschhorn. See Theorem 5.1.1 of \cite{hirschhorn:2003}. It is sufficient to show that the class of $C$-equivalences of pre-cooperads is equal to the class of $K$-colocal equivalences for some set $K$ of pre-cooperads.

For $n \in \mathbb{Z}$, define $J_n: \mathsf{Tree} \to \spectra$ as follows. For an $I$-tree $T$, let
\[ J_n(T) := \bar{w}(T) \smsh (\Sigma_I)_+ \smsh S^{n}_c \]
where $\Sigma_I$ is the symmetric group on the set $I$, and $S^{n}_c$ is a cofibrant model for the $n$-sphere spectrum in $\spectra$. A morphism $\sigma: T \to T'$ in $\mathsf{Tree}$ determines a map $\bar{w}(T) = \bar{w}(\sigma_*T) \to \bar{w}(T')$ and the bijection $\sigma: I \to I'$ determines a map $\Sigma_I \to \Sigma_{I'}$. Combining these, we get the necessary map $J_n(T) \to J_n(T')$. We take $K$ to be the set of free pre-cooperads $\{\mathbb{F}J_n \; | \; n \in \mathbb{Z} \}$.

For a pre-cooperad $Q$, we have
\[ \begin{split}
    \Hom_{\mathsf{PreCooperad}}(\mathbb{F}J_n,Q) &\isom \Hom_{\mathsf{Tree}}(J_n,Q) \\
        &\isom \prod_{k = 1}^{\infty} \Hom_{T \in \mathsf{T}(k)}(\bar{w}(T) \smsh S^{n}_c,Q(T)) \\
        &\isom \prod_{k = 1}^{\infty} \Hom(S^n_c,CQ(k))
\end{split} \]
A morphism of pre-cooperads $Q \to Q'$ is therefore a $K$-colocal equivalence if and only if the maps
\[ \Hom(S^n_c,CQ(k)) \to \Hom(S^n_c,CQ'(k)) \]
are weak equivalences of simplicial sets for all $k \geq 1, n \in \mathbb{Z}$. This is the case if and only if $CQ(k) \to CQ'(k)$ is a weak equivalence of spectra for all $k \geq 1$, that is, if and only if $Q \to Q'$ is a $C$-equivalence.
\end{proof}

\begin{definition}[Left adjoint to the cobar construction] \label{def:left-adjoint}
Given $I$-trees $T,U$, we define a pointed simplicial set $\bar{w}(T;U)$ by:
\[ \bar{w}(T;U) := \begin{cases}
  \Smsh_{u \in U} \bar{w}(T_u) & \text{if $U \leq T$}; \\
  \quad * & \text{otherwise}.
\end{cases} \]
For $U \leq U'$, we have
\[ \bar{w}(T;U) \to \bar{w}(T;U') \]
given, if $U' \leq T$, by smashing together, over $u \in U$, the maps
\[  \nu: \bar{w}(T_u) \to \Smsh_{u' \in (U')_u} \bar{w}(T_{u'}) \]
of Definition \ref{def:bar}.

If $T \leq T'$, we have
\[ \bar{w}(T;U) \to \bar{w}(T';U) \]
given, if $U \leq T$, by smashing together, over $u \in U$, the maps
\[ \iota: \bar{w}(T_u) \to \bar{w}(T'_u). \]
These maps make $\bar{w}(-;-)$ into a functor $\mathsf{T}(I) \times \mathsf{T}(I) \to \sset$.

We also have isomorphisms
\[ \mu_{T;U,T';U'}: \bar{w}(T;U) \smsh \bar{w}(T';U') \arrow{e,t}{\isom} \bar{w}(T \cup_i T'; U \cup_i U'). \]

Now let $P$ be a reduced operad. We define a pre-cooperad $\mathbb{B}P$ by
\[ \mathbb{B}P(T) := \bar{w}(T;U)^{\mathsf{rev}} \smsh_{U \in \mathsf{T}(I)} P(U). \]
The pre-cooperad structure maps come from combining the isomorphisms $\mu_{T;U,T';U'}$ above with the isomorphisms $P(U) \smsh P(U') \arrow{e,t}{\isom} P(U \cup_i U')$. These constructions determine a functor
\[ \mathbb{B}: \mathsf{Operad} \to \mathsf{PreCooperad}. \]
\end{definition}

\begin{lemma} \label{lem:adjunction}
The functor $\mathbb{B}: \mathsf{Operad} \to \mathsf{PreCooperad}$ is left adjoint to the cobar construction $C: \mathsf{PreCooperad} \to \mathsf{Operad}$, and $(\mathbb{B},C)$ is a Quillen adjunction between the projective model structure on $\mathsf{Operad}$ and the $C$-model structure of Proposition \ref{prop:model-precooperad} on $\mathsf{PreCooperad}$.
\end{lemma}
\begin{proof}
A map of operads $\phi: P \to CQ$ gives us maps
\[ P(I) \to \Map_{T \in \mathsf{T}(I)}(\bar{w}(T)^{\mathsf{rev}},Q(T)) \]
which are adjoint to maps, natural in $T$,
\[ \tag{*} \phi_T: \bar{w}(T)^{\mathsf{rev}} \smsh P(I) \to Q(T) \]
If $U \leq T$, we can smash together maps of the form (*) for each $T_u$ to get
\[ \phi_{T,U}: \bar{w}(T;U)^{\mathsf{rev}} \smsh P(U) \to \Smsh_{u \in U} Q(T_u) \to Q(T). \]
If $U \nleq T$, we take $\phi_{T,U}$ to be the trivial map. We now claim that the $\phi_{T,U}$ determine a map $\phi^\#_T: \mathbb{B}P(T) \to Q(T)$. To see this we must check that the following diagram commutes
\[ \begin{diagram}
  \node[2]{\bar{w}(T;U)^{\mathsf{rev}} \smsh P(U)} \arrow{se,t}{\phi_{T,U}} \\
  \node{\bar{w}(T;U)^{\mathsf{rev}} \smsh P(U')} \arrow{ne} \arrow{se} \node[2]{Q(T)} \\
  \node[2]{\bar{w}(T;U')^{\mathsf{rev}} \smsh P(U')} \arrow{ne,b}{\phi_{T,U'}}
\end{diagram} \]
when $U \leq U'$. This follows from the hypothesis that $\phi$ is a map of operads. Finally, it is easy to check that the $\phi^\#_T$ form a map of pre-cooperads $\phi^\#: \mathbb{B}P \to Q$.

Conversely, a map of pre-cooperads $\psi: \mathbb{B}P \to Q$ determines maps
\[ \psi_{T,\tau_I}: \bar{w}(T,\tau_I)^{\mathsf{rev}} \smsh P(I) \to Q(T) \]
and hence, since $\bar{w}(T,\tau_I) = \bar{w}(T)$, maps
\[ \psi^\#_T: P(I) \to \Map(\bar{w}(T)^{\mathsf{rev}},Q(T)). \]
These combine to form maps
\[ \psi^\#_I: P(I) \to CQ(I) \]
which make up a map of operads $\psi^\#: P \to CQ$.

The cobar construction $C$ preserves all weak equivalences by definition of the $C$-model structure. If $Q \to Q'$ is a fibration of pre-cooperads, then in particular, $Q(-) \to Q'(-)$ is a Reedy fibration of $\mathsf{T}(I)$-indexed diagrams of spectra, for each $I$ (since these too are determined termwise, see the proof of Proposition \ref{prop:bar-cobar-homotopy}). In \ref{prop:bar-cobar-homotopy} we also saw that $\bar{w}(-)^{\mathsf{rev}}$ is a Reedy cofibrant diagram. It follows by \cite[18.4.11]{hirschhorn:2003} that $\Map_{T}(\bar{w}(T)^{\mathsf{rev}},-)$ takes Reedy fibrations to fibrations of spectra. Hence $C$ also preserves fibrations and hence trivial fibrations. Thus $(\mathbb{B},C)$ is a Quillen pair.
\end{proof}

\begin{remark} \label{rem:cobar-adjoints}
In the context of operads of chain complexes over a commutative ring $R$, Getzler and Jones \cite{getzler/jones:1994} show that the cobar and bar constructions form an adjunction between categories of operads and cooperads in which the cobar construction $C$ is the \emph{left} adjoint. The constructions of this section can be applied to the algebraic case and so $C$ is also a right adjoint as a functor from pre-cooperads to operads. Note that limits and colimits of cooperads are very different depending on whether they are calculated in the category of pre-cooperads or cooperads.
\end{remark}

\begin{definition}[The co-$W$-construction] \label{def:co-$W$-construction}
In order to show that the pair $(\mathbb{B},C)$ is a Quillen \emph{equivalence} we need a version of the $W$-construction for pre-cooperads.

Let $T \leq U$ be $I$-trees. We then write
\[ \Delta[U;T] := \prod_{\mathsf{edge}(U) - \mathsf{edge}(T)} \Delta[1]. \]
This is a product of copies of the reversed simplicial interval indexed by those edges of $U$ that are contracted in $T$. If $T \leq U \leq U'$, we have a map of simplicial sets
\[ i_{U,U'}: \Delta[U;T] \to \Delta[U';T] \]
given by assigning value $0$ to the edges of $U'$ that are contracted in $U$. Alternatively, if $T' \leq T \leq U$, we have a map
\[ j_{T,T'}: \Delta[U;T] \to \Delta[U;T'] \]
given by assigning value $0$ to the edges of $U$ that are contracted in $T'$ but not $T$. Both of these maps are the inclusions of faces in a simplicial cube. We also have relabelling isomorphisms
\[ \sigma_\#: \Delta[U;T] \arrow{e,t}{\isom} \Delta[\sigma_*U,\sigma_*T] \]
for a bijection $\sigma: I \arrow{e,t}{\isom} I'$ and a grafting map
\[ \mu_i: \Delta[U \cup_i U'; T \cup_i T'] \arrow{e,t}{\isom} \Delta[U;T] \times \Delta[U';T']. \]
Both the $\sigma_\#$ and $\mu_i$ are natural with respect to the $i_{U,U'}$ and $j_{T,T'}$.

Let $Q$ be a pre-cooperad. We define the \emph{co-$W$-construction on $Q$} to be the pre-cooperad $W^cQ$ given on $T \in \mathsf{T}(I)$ by
\[ W^cQ(T) := \Map_{T \leq U}(\Delta[U;T]^{\mathsf{rev}}_+,Q(U)). \]
The object $W^cQ(T)$ is an `end' for diagrams indexed by the subcategory of $\mathsf{T}(I)$ consisting of trees $U$ with $T \leq U$. The maps $j_{U,U'}$ above make $\Delta[-;T]^{\mathsf{rev}}_+$ into such a diagram of simplicial sets and the pre-cooperad $Q$ restricts to such a diagram of spectra.

The maps $j_{T,T'}$ and $\sigma_\#$ above determine maps $W^cQ(T') \to W^cQ(T)$ that make $W^cQ$ into a functor $\mathsf{Tree} \to \spectra$. Combining the maps $\mu_i$ above with the structure maps for the pre-cooperad $Q$, we obtain maps
\[ W^cQ(T) \smsh W^cQ(T') \to W^cQ(T \cup_i T') \]
that make $W^cQ$ into a pre-cooperad.
\end{definition}

\begin{definition}[The co-$W$-resolution map] \label{def:co-W-map}
We construct a natural map of pre-cooperads
\[ \eta^*: Q \to W^cQ. \]
For $T \leq U$, we have maps
\[ \Delta[U;T]^{\mathsf{rev}}_+ \smsh Q(T) \to Q(U) \]
given by contracting the cube $\Delta[U;T]$ to a point and combining with the map $Q(T) \to Q(U)$. These are natural in $T$ and $U$. Their adjoints
\[ Q(T) \to \Map(\Delta[U;T]^{\mathsf{rev}}_+,Q(U)) \]
combine to form a map
\[ \eta^*_T: Q(T) \to W^cQ(T) \]
and these make up a map of pre-cooperads $Q \to W^cQ$.
\end{definition}

\begin{lemma} \label{lem:Q-WQ}
The map $\eta^*: Q \to W^cQ$ of Definition \ref{def:co-W-map} is a strict weak equivalence of pre-cooperads.
\end{lemma}
\begin{proof}
We construct a simplicial homotopy inverse to $\eta^*_T$ from maps
\[ \zeta^*_T: W^cQ(T) \to Q(T). \]
These are given by projecting from $W^cQ(T)$ onto the term $U = T$ and using the isomorphism
\[ \Map(\Delta[T;T]^{\mathsf{rev}}_+,Q(T)) \arrow{e,t}{\isom} Q(T). \]
The composite
\[ Q(T) \arrow{e,t}{\eta^*_T} W^cQ(T) \arrow{e,t}{\zeta^*_T} Q(T) \]
is the identity, and the composite
\[ W^cQ(T) \arrow{e,t}{\zeta^*_T} Q(T) \arrow{e,t}{\eta^*_T} W^cQ(T) \]
is simplicially homotopic to the identity. The homotopy is made from deformation retractions of $\Delta[U;T]$ onto the point where all edges have value $0$. These retractions can be chosen to be natural in $U$. It follows that $\eta^*_T:Q(T) \to W^cQ(T)$ is a weak equivalence for all $T$ as required.

Notice that the maps $\zeta^*_T$ are natural in $T$, that is, we have an objectwise weak equivalence $\zeta^*: W^cQ \to Q$ in $\spectra^{\mathsf{Tree}}$. However, $\zeta^*$ is not in general a map of pre-cooperads.
\end{proof}

We now relate the co-$W$-construction to the bar-cobar construction by constructing a map of pre-cooperads
\[ \theta^*: BCQ \to W^cQ. \]
This is more-or-less dual to the construction of the map $\theta: WP \to CBP$ in Definition \ref{def:W-CB}.

\begin{definition}[The map of pre-cooperads $BCQ \to W^cQ$] \label{def:BC-W}
For an $I$-tree $T$ and any operad $P$ we have
\[ \begin{split} BP(T) &= \Smsh_{t \in T} \bar{w}(U_t) \smsh_{U_t \in \mathsf{T}(I_t)} P(U_t) \\ &\isom \left[\Smsh_{t \in T} \bar{w}(U_t)\right] \smsh_{T \leq U} P(U). \end{split} \]
In particular, this gives
\[ BCQ(T) \isom \left[\Smsh_{t \in T} \bar{w}(U_t)\right] \smsh_{T \leq U} \Smsh_{u \in U} \Map_{V_u \in \mathsf{T}(I_u)}(\bar{w}(V_u^{\mathsf{rev}}),Q(V_u)) \]
which has a natural map, for each $V \geq U$ to
\[ \left[\Smsh_{t \in T} \bar{w}(U_t)\right] \smsh_{T \leq U} \Map \left(\Smsh_{u \in U}\bar{w}(V_u)^{\mathsf{rev}},Q(V) \right). \]
To define a map from here to
\[ W^cQ(T) = \Map_{T \leq V}(\Delta[V;T]^{\mathsf{rev}}_+,Q(V)) \]
it is sufficient to give, for each $U$ with $T \leq U \leq V$, a map of simplicial sets
\[ \theta^*_{T,U,V}: \Delta[V;T]^{\mathsf{rev}}_+ \smsh \Smsh_{t \in T} \bar{w}(U_t) \to \Smsh_{u \in U} \bar{w}(V_u)^{\mathsf{rev}}. \]
We get such maps by smashing together, over $t \in T$, the maps
\[ \theta_{V_t,U_t}^{\mathsf{rev}}: \Delta[V_t]^{\mathsf{rev}}_+ \smsh \bar{w}(U_t) \to \Smsh_{u \in U_t} \bar{w}(V_u)^{\mathsf{rev}} \]
of Definition \ref{def:W-CB} and noticing that
\[ \Delta[V;T] \isom \prod_{t \in T} \Delta[V_t]. \]

The necessary checks here are essentially the same as those used in the construction of the map $\theta: WP \to CBP$ in Definition \ref{def:W-CB}. They amount to checking that the maps $\theta^*_{T,U,V}$ are natural in $T,U,V$, and respect the grafting maps
\[ \Delta[V;T] \times \Delta[V';T'] \arrow{e,t}{\isom} \Delta[V \cup_i V';T \cup_i T'], \]
and decomposition maps
\[ \bar{w}(U \cup_i U') \to \bar{w}(U) \smsh \bar{w}(U') \]
and
\[ \bar{w}(V \cup_i V') \to \bar{w}(V) \smsh \bar{w}(V'). \]
Together these ensure that we have a well-defined map of pre-cooperads
\[ \theta^*: BCQ \to W^cQ. \]
\end{definition}

\begin{lemma} \label{lem:BC-W-weq}
For a pre-cooperad $Q$, there is a natural $C$-equivalence of pre-cooperads
\[ \tilde{\theta}^*: B\tilde{C}Q \weq W^cQ. \]
\end{lemma}
\begin{proof}
The morphism $\tilde{\theta}^*$ is obtained by composing the map $\theta^*$ of Definition \ref{def:BC-W} with the termwise-cofibrant replacement map $\tilde{C}Q \to CQ$. To check that $\tilde{\theta}^*$ is a $C$-equivalence of pre-cooperads, we have to show that it induces a weak equivalence $CB\tilde{C}Q \to CW^cQ$ of operads, that is, of symmetric sequences. We have the following diagram of symmetric sequences
\[ \begin{diagram}
  \node{\tilde{C}Q} \arrow{e,t}{\sim} \arrow{s,lr}{\zeta}{\sim} \node{CQ} \arrow{s,lr}{\zeta}{\sim} \\
  \node{W\tilde{C}Q} \arrow{e,t}{\sim} \arrow{s,lr}{\theta}{\sim} \node{WCQ} \arrow{s,l}{\theta} \\
  \node{CB\tilde{C}Q} \arrow{e} \node{CBCQ} \arrow{e,t}{C\theta^*} \node{CW^cQ} \arrow{e,tb}{C\zeta^*}{\sim} \node{CQ}
\end{diagram} \]
The map $C\zeta^*$ is given by applying the cobar construction $C$ to the natural transformation $\zeta^*: W^cQ \to Q$ of Lemma \ref{lem:Q-WQ}. Since $\zeta^*$ is an objectwise weak equivalence between Reedy fibrant diagrams $\mathsf{T}(I) \to \spectra$, the map $C\zeta^*$ is a weak equivalence of symmetric sequences.

It is now sufficient to show that the composite
\[ \dgTEXTARROWLENGTH=3em \psi: CQ \arrow{e,t}{\zeta} WCQ \arrow{e,t}{\theta} CBCQ \arrow{e,t}{C\theta^*} CW^cQ \arrow{e,t}{C\zeta^*} CQ \]
is a weak equivalence of symmetric sequences. Following through the definitions, we can explicitly describe the map $\psi$ as follows. For a nonempty finite set $I$, the map
\[ \psi_I: \Map_{T \in \mathsf{T}(I)}(\bar{w}(T),Q(T)) \to \Map_{T \in \mathsf{T}(I)}(\bar{w}(T),Q(T)) \]
is determined by the natural maps
\[ \psi_T: \bar{w}(T) \to \bar{w}(T) \]
that are trivial if $T \neq \tau_I$ and the identity if $T = \tau_I$. We then see that $\psi_I$ is a weak equivalence of spectra by noticing that the $\psi_T$, considered as a map of $\mathsf{T}(I)$-indexed diagrams, form an objectwise weak equivalence between Reedy cofibrant objects. (For $T \neq \tau_I$, $\psi_T$ is a weak equivalence of simplicial sets because $\bar{w}(T)$ is contractible.) Therefore, the map induced by the $\psi_T$ by applying
\[ \Map_{T \in \mathsf{T}(I)}(-,Q(T)) \]
is a weak equivalence of spectra.
\end{proof}

Combining \ref{lem:BC-W-weq} and \ref{lem:Q-WQ}, we see that $B\tilde{C}Q$ is naturally C-weakly equivalent to $Q$. It now follows that the derived functors of $B$ and $C$ are inverse equivalences. In turn this allows us to deduce the main theorem of this paper.

\begin{theorem} \label{thm:Quillen-equiv}
The functors $\mathbb{B}: \mathsf{Operad} \to \mathsf{PreCooperad}$ and $C: \mathsf{PreCooperad} \to \mathsf{Operad}$ form a Quillen equivalence between the projective model structure on the category of operads and the $C$-model structure on the category of pre-cooperads.
\end{theorem}
\begin{proof}
Combining Lemmas \ref{lem:BC-W-weq} and \ref{lem:Q-WQ}, we see that the cobar construction $C$ has a left inverse on the homotopy category. Since Theorem \ref{thm:W-CB} tells us $C$ has a right inverse, it follows that $C$ induces an equivalence of homotopy categories. Therefore the Quillen adjunction $(\mathbb{B},C)$ is a Quillen equivalence.
\end{proof}

The $C$-equivalences in the category of pre-cooperads are still somewhat mysterious as they are defined indirectly via the cobar construction. In order to add significance to Theorem \ref{thm:Quillen-equiv} we now interpret the homotopy category of pre-cooperads more intrinsically. The key notion here is that of a `quasi-operad' defined below. We show below that the cofibrant pre-cooperads in the $C$-model structure are termwise-cofibrant quasi-operads, and that a map of quasi-cooperads is a $C$-equivalence if and only if it is a strict weak equivalence. It follows that the homotopy category of the $C$-model structure can be identified with the homotopy category of termwise-cofibrant quasi-cooperads, with respect to the strict weak equivalences.

\begin{definition}[Quasi-cooperads] \label{def:quasi-cooperad}
We say that a pre-cooperad $Q$ is a \emph{quasi-cooperad} if the maps
\[ m_{T,i,U}: Q(T) \smsh Q(U) \weq Q(T \cup_i U) \]
are weak equivalences of spectra for all $T,i,U$. Thus a quasi-cooperad is `almost' an actual cooperad, except that the putative decomposition maps are only defined up to inverse weak equivalences. For example we have maps
\[ Q(\tau_3) \to Q(\tau_2 \cup_i \tau_2) \lweq Q(\tau_2) \smsh Q(\tau_2). \]
In particular, any cooperad is a quasi-cooperad.
\end{definition}

The key fact about quasi-cooperads is the following.

\begin{prop} \label{prop:quasi-weq}
A map $Q \to Q'$ of quasi-cooperads is a $C$-equivalence if and only if it is a strict weak equivalence.
\end{prop}
\begin{proof}
First note that strict weak equivalences of pre-cooperads are always $C$-equivalences so there is only one direction to do here. Suppose therefore that $\phi: Q \to Q'$ is a $C$-equivalence. We may assume without loss of generality that $Q$ and $Q'$ are cofibrant in the strict model structure on pre-cooperads. Hence by Lemma \ref{lem:strictly-cofibrant-precooperad} each $Q(T)$ is cofibrant in $\spectra$.

We prove that $\phi_T: Q(T) \to Q'(T)$ is a weak equivalence by induction on $|I|$ where $T \in \mathsf{T}(I)$. For $|I| = 2$, the only possible $T$ is $\tau_I$. We have
\[ CQ(I) \isom \Omega Q(\tau_I) \]
and hence, since $\spectra$ is a stable model category, $C\phi_I: CQ(I) \to CQ'(I)$ is a weak equivalence if and only if $\phi_{\tau_I}: Q(\tau_I) \to Q'(\tau_I)$ is. This completes the base case of the induction.

Now suppose that $\phi_T: Q(T) \to Q'(T)$ is a weak equivalence whenever $T \in \mathsf{T}(I)$ with $|I| < n$. Suppose that $|I| = n$ and consider $T \in \mathsf{T}(I)$. If $T \neq \tau_I$, then we have $T = U \cup_j V$ for some trees $U \in \mathsf{T}(J)$ and $V \in \mathsf{T}(K)$ where $I = J \cup_j K$, and $|J|,|K| < n$. By the induction hypothesis, and since $Q$ and $Q'$ are termwise-cofibrant quasi-cooperads, we then have the following diagram
\[ \begin{diagram}
  \node{Q(U) \smsh Q(V)} \arrow{e,t}{\sim} \arrow{s,lr}{\sim}{\phi_U \smsh \phi_V} \node{Q(T)} \arrow{s,r}{\phi_T} \\
  \node{Q'(U) \smsh Q'(V)} \arrow{e,t}{\sim} \node{Q'(T)}
\end{diagram} \]
which implies that $\phi_T: Q(T) \to Q'(T)$ is a weak equivalence.

Finally consider the case $T = \tau_I$ and make the following definitions. Define $Q_1: \mathsf{T}(I) \to \spectra$ by
\[ Q_1(T) := \begin{cases} Q(T) & \text{if $T \neq \tau_I$}; \\ \; \; * & \text{if $T = \tau_I$}; \end{cases} \]
and $Q_0: \mathsf{T}(I) \to \spectra$ by
\[ Q_0(T) := \begin{cases} \; \; * & \text{if $T \neq \tau_I$}; \\ Q(\tau_I) & \text{if $T = \tau_I$}. \end{cases} \]
We then have a strict fibre sequence of maps of $\mathsf{T}(I)$-diagrams:
\[ Q_1 \to Q \to Q_0 \]
with $Q \to Q_0$ a Reedy fibration. (Recall that Reedy fibrations are precisely the objectwise fibrations for $\mathsf{T}(I)$-indexed diagrams.)

Applying $\Map_{T \in \mathsf{T}(I)}(\bar{w}(T),-)$ to this sequence we get a fibre sequence of spectra, which we can easily identify as
\[ \Map_{T}(\bar{w}(T),Q_1(T)) \to CQ(I) \to \Omega Q(\tau_I) \]
where $CQ(I) \to \Omega Q(\tau_I)$ is a fibration in $\spectra$. In particular, this is a homotopy-fibre sequence. We now have a diagram
\[ \begin{diagram}
  \node{\Map_{T}(\bar{w}(T),Q_1(T))} \arrow{e} \arrow{s,lr}{\sim}{\phi} \node{CQ(I)} \arrow{e} \arrow{s,lr}{\sim}{C\phi_I} \node{\Omega Q(\tau_I)} \arrow{s,r}{\Omega \phi_{\tau_I}} \\
  \node{\Map_{T}(\bar{w}(T),Q'_1(T))} \arrow{e} \node{CQ'(I)} \arrow{e} \node{\Omega Q'(\tau_I)}
\end{diagram} \]
where the rows are homotopy-fibre sequences of spectra. The left-hand vertical map is an equivalence by the case $T \neq \tau_I$, since $\phi: Q_1 \to Q'_1$ is an objectwise weak equivalence between Reedy fibrant diagrams. The centre vertical map is a weak equivalence by the hypothesis that $\phi: Q \to Q'$ is a $C$-equivalence. Thus we deduce that $\Omega\phi_{\tau_I}$ is a weak equivalence and hence so is $\phi_{\tau_I}$, again using the fact that $\spectra$ is a stable model category.
\end{proof}

\begin{cor} \label{cor:BC-W}
If $Q$ is a quasi-cooperad, then the morphism
\[ \tilde{\theta}^*: B\tilde{C}Q \to W^cQ \]
of Lemma \ref{lem:BC-W-weq} is a strict weak equivalence between quasi-cooperads. Thus for a quasi-cooperad $Q$, there is a zigzag of strict weak equivalences of quasi-cooperads
\[ B\tilde{C}Q \homeq Q. \]
\end{cor}
\begin{proof}
Lemma \ref{lem:Q-WQ} tells us that $W^cQ$ is a quasi-cooperad, and $B\tilde{C}Q$ is a cooperad, hence a quasi-cooperad. This then follows from Proposition \ref{prop:quasi-weq} and Lemma \ref{lem:BC-W-weq}.
\end{proof}

\begin{cor} \label{cor:quasi}
A pre-cooperad $Q$ is a quasi-cooperad if and only if $Q$ is strictly weakly equivalent to a termwise-cofibrant cooperad. \qed
\end{cor}

\begin{prop} \label{prop:cofibrant-precooperad}
A pre-cooperad $Q$ is cofibrant in the $C$-model structure if and only if it is a strictly cofibrant quasi-cooperad.
\end{prop}
\begin{proof}
By \cite[5.1.6]{hirschhorn:2003}, the cofibrant pre-cooperads in the $C$-model structure are the `$K$-cellular' objects (where $K$ is the set of free pre-cooperads defined in Proposition \ref{prop:model-precooperad}). The class of $K$-cellular objects is the smallest class of strictly cofibrant pre-cooperads that contains $K$ and is closed under strict weak equivalences and homotopy colimits. We therefore claim that the $K$-cellular pre-cooperads are precisely the strictly-cofibrant quasi-cooperads.

First note that each $\mathbb{F}J_n$ is a quasi-cooperad since $J_n(T)$ is contractible for $T \neq \tau_I$. (Compare with Remark \ref{rem:free-precooperad}.)

Next we claim that a homotopy colimit of strictly-cofibrant quasi-cooperads is again a quasi-cooperad. To see this, notice that we have strict weak equivalences
\[ \hocolim Q_\alpha \homeq \hocolim B\tilde{C}Q_\alpha \homeq B(\hocolim \tilde{C}Q_\alpha). \]
The first equivalence follows from the fact that a quasi-cooperad $Q_\alpha$ is strictly weakly equivalent to $B\tilde{C}Q_\alpha$, by Corollary \ref{cor:BC-W}, and the fact that the homotopy colimit preserves objectwise strict weak equivalences. The second equivalence exists because the bar construction $B$ preserves homotopy colimits (it is equivalent to $\mathbb{B}$ on cofibrant operads, and $\mathbb{B}$ is a left Quillen functor with respect to the strict model structure on pre-cooperads). But we have now shown that $\hocolim Q_\alpha$ is strictly weakly equivalent to a cooperad, hence is a quasi-cooperad.

Finally, the class of quasi-cooperads is closed under strict weak equivalence, so we deduce that the class of $K$-cellular pre-cooperads is contained in the class of quasi-cooperads.

Conversely, let $Q$ be any strictly-cofibrant quasi-cooperad. Let $\tilde{Q}$ be a cofibrant replacement for $Q$ in the $C$-model structure. We have just shown that $\tilde{Q}$ is a quasi-cooperad, so the $C$-equivalence $\tilde{Q} \weq Q$ must be a strict weak equivalence by Proposition \ref{prop:quasi-weq}. But $\tilde{Q}$ is $K$-cellular so $Q$ must also be.
\end{proof}

We can now interpret the homotopy category of pre-cooperads directly in terms of quasi-cooperads and termwise weak equivalences.

\begin{cor} \label{cor:precooperad}
The homotopy category of pre-cooperads and $C$-equivalences is equivalent to the homotopy category of termwise-cofibrant quasi-cooperads and strict weak equivalences.
\end{cor}
\begin{proof}
The homotopy category of pre-cooperads is equivalent to that of the cofibrant-fibrant pre-cooperads, that is the strictly cofibrant quasi-cooperads. By Lemma \ref{lem:strictly-cofibrant-precooperad}, a strictly-cofibrant quasi-cooperad is termwise-cofibrant. For quasi-cooperads $C$-equivalences are always strict.
\end{proof}

\begin{remark} \label{rem:quasi-rigid}
Any quasi-cooperad $Q$ has a `rigidification', that is, is strictly weakly equivalent to an actual cooperad, namely $B\tilde{C}Q$. However, there is no reason to believe that morphisms in the homotopy category of quasi-cooperads can be realized by zigzags of maps of actual cooperads. For example, even when $Q$ is a cooperad, the equivalence between the cooperads $B\tilde{C}Q$ and $Q$ passes through the quasi-cooperad $W^cQ$.
\end{remark}

We conclude this section by noting that the various functors we have between operads and (pre-) cooperads are simplicial. In particular this means that they determine equivalences of derived mapping spaces, not just equivalences of homotopy categories.

\begin{lemma} \label{lem:BC-simplicial}
The functors $\mathbb{B}: \mathsf{Operad} \to \mathsf{PreCooperad}$, $B: \mathsf{Operad} \to \mathsf{PreCooperad}$ and $C: \mathsf{PreCooperad} \to \mathsf{Operad}$ are simplicial with respect to the standard simplicial structures on these categories. The adjunction between $\mathbb{B}$ and $C$ is simplicial in the sense that there are natural isomorphisms of simplicial sets
\[ \Hom_{\mathsf{Operad}}(P,CQ) \isom \Hom_{\mathsf{PreCooperad}}(\mathbb{B}P,Q). \]
for an operad $P$ and pre-cooperad $Q$.
\end{lemma}
\begin{proof}
To show this we take advantage of the fact that the simplicial \emph{cotensoring} in both the categories $\mathsf{Operad}$ and $\mathsf{PreCooperad}$ is done termwise using the diagonal on a pointed simplicial set. To show that $B$ is simplicial, it is sufficient to construct, for $X \in \sset$ and $P \in \mathsf{Operad}$, natural maps of pre-cooperads
\[ B\Map(X,P) \to \Map(X,BP), \]
that reduce to the identity when $X = S^0$. At a nonempty finite set $I$, we define such a map by
\[ \begin{split}
  \bar{w}(T) \smsh_{T \in \mathsf{T}(I)} \Smsh_{t \in T} \Map(X,P(I_t))
    &\to \bar{w}(T) \smsh_{T \in \mathsf{T}(I)} \Map(\Smsh_{t \in T} X,P(T)) \\
    &\to \bar{w}(T) \smsh_{T \in \mathsf{T}(I)} \Map(X,P(T)) \\
    &\to \Map(X,\bar{w}(T) \smsh_{T \in \mathsf{T}(I)} P(T))
\end{split} \]
where the first map smashes together the spectra $\Map(X,P(I_t))$, the second uses the diagonal $X \to \Smsh_{t \in T} X$, and the third comes from the appropriate adjunctions. The reader can check that these maps do indeed determine a natural map of pre-cooperads.

The construction is virtually identical for $\mathbb{B}$ and similar for $C$. The existence of the claimed isomorphism then follows from the form of the simplicial structures for $\mathbb{B}$ and $C$.
\end{proof}

\section{Koszul duality for termwise-finite operads} \label{sec:koszul}

The first description of bar-duality for operads of chain complexes was by Ginzburg and Kapranov in \cite{ginzburg/kapranov:1994}. They described this theory purely in terms of operads, using linear duality to avoid any mention of cooperads. It is convenient to have a corresponding description in the case of spectra.

To describe bar-duality using operads alone, we employ Spanier-Whitehead duality to convert cooperads into operads. This plays the role that linear duality does in the algebraic setting. Thus we define the `derived Koszul dual' $KP$ of an operad $P$ to be the operad formed by the Spanier-Whitehead dual of the cooperad $BP$. Our main task in this section is to establish that, subject to finiteness conditions, the double dual $K(K(P))$ is equivalent to $P$. We then also see that $K$ preserves simplicial enrichments, homotopy colimits and suitably finite homotopy limits.

\begin{remark} \label{rem:koszul}
It should be pointed out that `Koszul' dual is not really an appropriate name for what we are calling $KP$. As originally described by Priddy \cite{priddy:1970}, for algebras, and Ginzburg-Kapranov \cite{ginzburg/kapranov:1994}, for operads, the Koszul dual in the algebraic setting is a certain minimal model for the (dual of the) bar construction. It is much smaller than, but quasi-isomorphic to, the full bar construction and, for example, helps us write down explicit resolutions for algebras over operads. In \cite{ching:2005a} and \cite{arone/ching:2009} we, used the term `Koszul dual' to denote the Spanier-Whitehead dual of the bar construction, which is the analogue of Ginzburg-Kapranov's `dg-dual'. Blumberg and Mandell \cite{blumberg/mandell:2009} have referred to a similar notion for ring spectra as the `derived Koszul dual'. In the topological case, there does not seem to be any obvious analogue of the Priddy/Ginzburg-Kapranov definition of the Koszul dual. Since we do not have a better name, we continue this usage.
\end{remark}

\begin{definition}[Spanier-Whitehead dual of a symmetric sequence] \label{def:dual-symseq}
For a spectrum $X$ we write
\[ \dual X := \Map(X,S) \]
where $S$ is the sphere spectrum. Here $\Map(-,-)$ denotes the internal mapping object (that is, the closed monoidal structure) in $\spectra$.

For a symmetric sequence $A$ in $\spectra$, we write $\dual A$ for the symmetric sequence given by
\[ (\dual A)(I) := \dual (A(I)). \]
A bijection $\sigma:I \arrow{e,t}{\isom} I'$ determines $(\dual A)(I) \to (\dual A)(I')$ by way of the map $\sigma^{-1}_*: A(I') \to A(I)$. If $A$ is reduced, then so is $\dual A$. We refer to $\dual A$ as the \emph{Spanier-Whitehead dual of $A$}.
\end{definition}

For spectra $X,Y$ there is a natural map
\[ \delta^*: \Map(X,S) \smsh \Map(Y,S) \to \Map(X \smsh Y,S). \]
that is a weak equivalence if $X$ and $Y$ are \emph{finite}, that is, weakly equivalent to finite cell spectra, and cofibrant. We use this to observe that the Spanier-Whitehead dual of a cooperad is an operad, and hence define the Koszul dual.

\begin{definition}[Derived Koszul dual of an operad] \label{def:koszul}
Let $Q$ be a reduced cooperad. Then we define a reduced operad structure on the Spanier-Whitehead dual $\dual Q$ with composition maps
\[ \dual Q(J) \smsh \Smsh_{j \in J} \dual Q(I_j) \arrow{e,t}{m} \dual \left(Q(J) \smsh \Smsh_{j \in J} \Q(I_j)\right) \arrow{e} \dual Q(I) \]
where the second map is induced by the cooperad structure map $Q(I) \to Q(J) \smsh \Smsh_{j \in J} Q(I_j)$. This construction determines a functor
\[ \dual: \mathsf{Cooperad} \to \mathsf{Operad}^{op} \]

If $P$ is a reduced operad, we define the \emph{derived Koszul dual} of $P$ to be the reduced operad
\[ KP := \dual BP. \]
This gives us a functor
\[ K: \mathsf{Operad} \to \mathsf{Operad}^{op}. \]
\end{definition}

\begin{definition}[Spanier-Whitehead dual of an operad] \label{def:dual-operad}
The Spanier-Whitehead dual of an operad is not in general a cooperad, but, suitably interpreted, it is a pre-cooperad, and, under finiteness conditions, a quasi-cooperad. For an operad $P$, we define $\dual P: \mathsf{Tree} \to \spectra$ by
\[ (\dual P)(T) := \dual P(T). \]
This conflicts with Definitions \ref{def:dual-symseq} and \ref{def:A(T)} in the sense that $\dual P$ has already been defined as a symmetric sequence and this definition of $(\dual P)(T)$ does not agree with that of \ref{def:A(T)}. We hope to avoid confusion on this point.

We make $\dual P$ into a pre-cooperad with structure maps
\[ \dual P(T) \smsh \dual P(U) \arrow{e,t}{m} \dual (P(T) \smsh P(U)) \isom \dual P(T \cup_i U). \]
This construction gives us a functor
\[ \dual: \mathsf{Operad} \to \mathsf{PreCooperad}^{op}. \]
If $P$ is a termwise-finite-cofibrant operad (see \ref{def:termwise-finite} below), then $P(T)$ is a finite-cofibrant spectrum for all $T$ and $\dual P$ is a quasi-cooperad.
\end{definition}

\begin{lemma} \label{lem:KP}
There is a natural isomorphism of operads
\[ KP \isom C \dual P. \]
\end{lemma}
\begin{proof}
This is a natural isomorphism
\[ \Map(\bar{w}(T) \smsh_{T \in \mathsf{T}(I)} P(T),S) \isom \Map_{T \in \mathsf{T}(I)}(\bar{w}(T),\Map(P(T),S)) \]
constructed from the usual adjunctions.
\end{proof}

\begin{remark}
The dual of a cooperad is an operad and the dual of an operad is a pre-cooperad, but we do not have dualizing functors in both directions between the same categories -- one cannot put an operad structure on the dual of an arbitrary pre-cooperad. Of course, we know that any pre-cooperad $Q$ is $C$-weakly equivalent to the cooperad $B\tilde{C}Q$ so the operad $\dual B \tilde{C}Q$ plays the role of the dual of $Q$.
\end{remark}

We now introduce the finiteness conditions that tell us when a double dual recovers the original object.

\begin{definition}[Termwise-finite operads] \label{def:termwise-finite}
We say that the symmetric sequence $A$ is \emph{termwise-finite} if, for each $I$, $A(I)$ is weakly equivalent to a finite cell $S$-module. An operad or cooperad is \emph{termwise-finite} if its underlying symmetric sequence is. Since it comes up a lot, we say that a symmetric sequence (or operad or cooperad) is \emph{termwise-finite-cofibrant} if it is both termwise-finite and termwise-cofibrant.
\end{definition}

\begin{definition}[Map from a cooperad to its double dual] \label{def:double-dual}
Let $Q$ be a cooperad. Then there is a natural map of pre-cooperads
\[ Q \to \dual \dual Q \]
defined as follows. For an $I$-tree $T$, we need to give a map
\[ Q(T) \to \Map \left(\Smsh_{t \in T} \Map(Q(I_t),S),S\right). \]
This is adjoint to the map
\[ \Smsh_{t \in T} \Map(Q(I_t),S) \to \Map(\Smsh_{t \in T} Q(I_t),S) \isom \Map(Q(T),S) \]
that smashes together the Spanier-Whitehead duals. Combining with a cofibrant replacement for the operad $\dual Q$, we get a map of pre-cooperads
\[ d: Q \to \dual \tilde{\dual}Q. \]
\end{definition}

\begin{lemma} \label{lem:double-dual}
Let $Q$ be a termwise-finite-cofibrant cooperad. Then the map
\[ d: Q \to \dual \tilde{\dual} Q \]
of Definition \ref{def:double-dual} is a strict weak equivalence of pre-cooperads.
\end{lemma}
\begin{proof}
We have a commutative diagram
\[ \begin{diagram}
  \node{Q(T)} \arrow{e,t}{d_T} \arrow{se,b}{\sim} \node{\Map\left(\Smsh_{t \in T} \tilde{\Map}(Q(I_t),S),S\right)} \arrow{s,r}{\sim} \\
  \node[2]{\Map(\tilde{\Map}(Q(T),S),S)}
\end{diagram} \]
where the tildes denote cofibrant replacement. The right-hand vertical map is an equivalence when each $Q(I_t)$ is finite-cofibrant, and the diagonal map is an equivalence when $Q(T)$ is finite-cofibrant. We deduce that $d_T$ is a weak equivalence as required.
\end{proof}

\begin{prop} \label{prop:CB-KK}
Let $P$ be a termwise-finite-cofibrant operad. Then there is a natural weak equivalence of operads
\[ d: CBP \to K \tilde{K}P \]
where $\tilde{K}P$ denotes a termwise-cofibrant replacement for the operad $KP$.
\end{prop}
\begin{proof}
By \cite[11.14]{arone/ching:2009}, $BP$ is a termwise-finite-cofibrant cooperad. Applying Lemma \ref{lem:double-dual} we get a strict weak equivalence
\[ BP \weq \dual \tilde{\dual}BP. \]
Applying the cobar construction, we get a weak equivalence of operads
\[ CBP \to C \dual \tilde{K}P \isom K\tilde{K}P \]
where the last isomorphism is from Lemma \ref{lem:KP}.
\end{proof}

\begin{theorem} \label{thm:KK}
For a termwise-finite-cofibrant operad $P$, there is a natural zigzag of equivalences of operads
\[ P \homeq K\tilde{K}P. \]
\end{theorem}
\begin{proof}
This follows from Proposition \ref{prop:CB-KK} and Theorem \ref{thm:W-CB}.
\end{proof}

\begin{cor}
The derived Koszul dual construction determines a contravariant equivalence between the homotopy category of termwise-finite-cofibrant reduced operads and itself.
\end{cor}

\begin{remark} \label{rem:dg-dual}
In Remark \ref{rem:Berger-Moerdijk} we noted that our definition of the bar construction corresponds exactly to the bar-cooperad for an operad $P$ of chain complexes, as described by Getzler and Jones \cite[\S2]{getzler/jones:1994}. Using a linear dual in place of the Spanier-Whitehead dual, our definition of $KP$ similarly corresponds exactly to the `dg-dual' $\mathbf{D}(P)$ of Ginzburg and Kapranov \cite[\S3]{ginzburg/kapranov:1994}.

Our construction of maps of operads $WP \to CBP \to KKP$ then correspond to maps
\[ WP \to CBP \to \mathbf{D}(\mathbf{D}(P)). \]
As we saw in Remark \ref{rem:Berger-Moerdijk}, the map $WP \to CBP$ is an isomorphism of operads of chain complexes, and the map $CBP \to \mathbf{D}\mathbf{D}P$ is an isomorphism as long as the terms in $P$ are finite-dimensional. Thus the zigzag of equivalences in Theorem \ref{thm:KK} reduces, in the algebraic case, to a single quasi-isomorphism
\[ \eta: \mathbf{D}(\mathbf{D}(P)) \to P \]
which, up to signs, is the same as that described in \cite[3.2.16]{ginzburg/kapranov:1994}.
\end{remark}

\begin{example} \label{ex:K-free-trivial}
It follows from work done in the proof of Lemma \ref{lem:omega-sigma} that trivial and free operads are Koszul dual to each other. For example, if $A$ is a termwise-finite-cofibrant symmetric sequence with the trivial operad structure, then $KA$ is weakly equivalent to the free operad on the symmetric sequence $\dual \Sigma A$. Conversely, the Koszul dual of the free operad on $A$ is weakly equivalent to the trivial operad $\Omega \dual A$.
\end{example}

We conclude this section by noting that Koszul duals preserves simplicial mapping spaces of operads, as well as homotopy colimits and finite homotopy limits.

\begin{thm} \label{thm:K-simplicial}
Suppose that $P$ and $P'$ are termwise-finite operads. Then there is a weak equivalence
\[ k_{P,P'}: \widetilde{\Hom}_{\mathsf{Operad}}(P,P') \weq \widetilde{\Hom}_{\mathsf{Operad}}(KP',KP) \]
where $\widetilde{\Hom}_{\mathsf{Operad}}(-,-)$ denotes the derived mapping space in the simplicial model category of operads.
\end{thm}
\begin{proof}
We saw in Lemma \ref{lem:BC-simplicial} that the bar and cobar functors are simplicial. The Spanier-Whitehead dual construction is also simplicial so the Koszul dual functor $K$ induces a morphism $k_{P,P'}$ as claimed. To see that this is a weak equivalence, we consider the following diagram (in the homotopy category of simplicial sets):
\[ \begin{diagram}
  \node{\widetilde{\Hom}_{\mathsf{Operad}}(P,P')} \arrow[2]{e,t}{k_{P,P'}} \arrow{s,l}{\sim} \node[2]{\widetilde{\Hom}_{\mathsf{Operad}}(KP',KP)} \arrow{s,r}{k_{KP',KP}} \\
  \node{\widetilde{\Hom}_{\mathsf{Operad}}(CBP,CBP')} \arrow{se,b}{\sim} \node[2]{\widetilde{\Hom}_{\mathsf{Operad}}(K\tilde{K}P,K\tilde{K}P')} \arrow{sw,b}{\sim} \\
  \node[2]{\widetilde{\Hom}_{\mathsf{Operad}}(CBP,K\tilde{K}P')}
\end{diagram} \]
Here the diagonal maps are induced by the equivalences $CBP' \to K\tilde{K}P'$ and $CBP \to K\tilde{K}P$ respectively, and the left-hand vertical map by the equivalences $P \homeq CBP$ and $P' \homeq CBP'$. Showing that the above diagram commutes amounts to checking that our natural maps $WP \to CBP \to KKP$ respect the simplicial structures on these functors. This is true and it shows that $k_{P,P'}$ is a right inverse to $k_{KP',KP}$. Replacing $P$ with $KP'$ and $P'$ with $KP$, we also see that $k_{KP',KP}$ has a left inverse. Hence all the maps in the above diagram are isomorphisms in the homotopy category.
\end{proof}

\begin{thm} \label{thm:K-hocolim}
Let $\{P_\alpha\}$ be a diagram of cofibrant reduced operads. Then the natural map
\[ K(\hocolim_{\alpha} P_\alpha) \to \holim_{\alpha} KP_{\alpha} \]
is a weak equivalence of operads. If the diagram is finite and each $P_\alpha$ is a termwise-finite operad, then the natural map
\[ \hocolim_{\alpha} KP_{\alpha} \to K(\holim_{\alpha} P_{\alpha}) \]
is also a weak equivalence of operads. (Note that the homotopy limits and colimits here are all formed within the category of operads, with respect to the simplicial tensoring and cotensoring. In particular, the homotopy limits are constructed termwise, but the homotopy colimits are not.)
\end{thm}
\begin{proof}
For the first part, it is enough to show that there is a natural weak equivalence
\[ \hocolim_{\alpha} BP_{\alpha} \weq B(\hocolim_{\alpha} P_{\alpha}) \]
where the homotopy colimit on the left-hand side is formed in the category of symmetric sequences, that is, termwise. Applying Spanier-Whitehead duality yields the claim since the homotopy limit in the category of operads is calculated termwise. To obtain the above equivalence, we use the fact that the bar construction of an operad $P$ can be identified, as a symmetric sequence, with the (termwise suspension of) the derived indecomposables. (Specifically, there is an isomorphism $BP \isom \Sigma \mathsf{indec}WP$.) The indecomposables functor $\mathsf{indec}$ is a left Quillen functor (with right adjoint the trivial operad functor) and so takes homotopy colimits of operads to homotopy colimits of symmetric sequences. The termwise suspension preserves homotopy colimits, so we have the required equivalence.

The second part follows from the first by applying Theorem \ref{thm:KK} to each $P_{\alpha}$ and to $\holim_{\alpha} P_{\alpha}$ (which is termwise-finite because a finite homotopy limit of homotopy-finite spectra is homotopy-finite).
\end{proof}

\begin{remark}
We have shown that derived Koszul duality determines a contravariant equivalence between the homotopy category of termwise-finite operads and itself. If $P$ is not termwise-finite, we do not expect to be able to recover $P$ from $KP$. Instead, one should replace $P$ with the filtered diagram of finite subcomplexes of a cellular replacement $\tilde{P}$. Applying $K$ objectwise to this diagram we obtain a `pro-operad', that is a cofiltered diagram of operads. One can then recover $P$ from this pro-operad by applying $K$ again objectwise and taking the homotopy colimit of the resulting filtered diagram.

One would hope that there exists a contravariant Quillen equivalence between the category of operads and an appropriate model structure on the category of pro-operads. Unfortunately, the analysis of Christensen-Isaksen \cite{christensen/isaksen:2004} for pro-spectra does not apply directly since the projective model structure on operads is not left proper. We therefore do not offer such a result.
\end{remark}

\section{Examples and Conjectures} \label{sec:examples}

We have one example of bar-cobar duality (beyond the duality between free and trivial operads) and various conjectures.

\begin{example}
Let $\mathsf{Com}$ denote the stable commutative operad described in \ref{ex:operads}. In \cite{ching:2005a} we proved that the terms of the derived Koszul dual $K(\mathsf{Com})$ are equivalent to Goodwillie's derivatives $\der_*I$ of the identity functor on based spaces -- see \cite{goodwillie:2003},\cite{johnson:1995},\cite{arone/mahowald:1999}. In \cite{arone/ching:2009}, with Greg Arone, we gave deeper insight into this result, identifying the operad $\mathsf{Com}$ with an operad $\der^*(\Sigma^\infty \Omega^\infty)$ formed by the Spanier-Whitehead duals of the Goodwillie derivatives of the functor
\[ \Sigma^\infty \Omega^\infty : \spectra \to \spectra. \]
It follows from Theorem \ref{thm:KK} that we also have
\[ K(\der_*I) \homeq \der^*(\Sigma^\infty \Omega^\infty). \]
Alternatively, if $\der_*(\Sigma^\infty \Omega^\infty)$ denotes the cooperad dual to $\der^*(\Sigma^\infty \Omega^\infty)$, then we have an equivalence of operads
\[ \der_*I \homeq C(\der_*(\Sigma^\infty \Omega^\infty)) \]
where the cooperad structure on $\der_*(\Sigma^\infty \Omega^\infty)$ comes from the comonad structure on the functor $\Sigma^\infty \Omega^\infty$.
\end{example}

\begin{conjecture} \label{conj:calculus}
Let $\cat{C}$ be a simplicial model category `in which one can do Goodwillie calculus'. (See, for example, Kuhn \cite{kuhn:2007}.) Let $\Sigma^\infty_\cat{C}$ and $\Omega^\infty_\cat{C}$ denote the `suspension spectrum' and `infinite loop-space' functors associated to a stabilization of the model category $\cat{C}$. The (imprecisely-stated) conjecture is that there is a cooperad
\[ \der_*(\Sigma^\infty_\cat{C} \Omega^\infty_\cat{C}) \]
whose terms are the Goodwillie derivatives (in a generalized sense) of $\Sigma^\infty_\cat{C} \Omega^\infty_\cat{C}$, and an equivalence of operads
\[ \der_*I_\cat{C} \homeq C(\der_*(\Sigma^\infty_\cat{C} \Omega^\infty_\cat{C}))  \]
where the terms of $\der_*I_\cat{C}$ are the Goodwillie derivatives (also in a generalized sense) of the identity functor on $\cat{C}$.
\end{conjecture}

To make proper sense of this conjecture, we need to work with \emph{coloured} operads of spectra, which are outside the scope of this paper. We do expect the form of our results to carry over to that setting though.

\begin{conjecture} \label{conj:P-alg}
Let $P$ be a reduced operad of spectra and let $\cat{C}$ be the category of $P$-algebras in $\spectra$. Then the objects described in Conjecture \ref{conj:calculus} satisfy
\[ \der_*I_\cat{C} \homeq P \]
and
\[ \der_*(\Sigma^\infty_\cat{C} \Omega^\infty_\cat{C}) \homeq BP. \]
\end{conjecture}

We also have conjectures for some of the other operads of spectra described in \ref{ex:operads}.

\begin{conjecture} \label{conj:discs}
Let $\Sigma^\infty_+ \mathcal{D}_n$ be the stable little $n$-discs operad (formed from the suspension spectra of the terms in the little $n$-discs operad of topological spaces). Then there is an equivalence of operads
\[ \tag{*} K(\Sigma^\infty_+ \mathcal{D}_n) \homeq s^{-n} \Sigma^\infty_+ \mathcal{D}_n. \]
The right-hand side here is an $n$-fold `operadic desuspension' of $\Sigma^\infty_+ \mathcal{D}_n$. Part of the conjecture is that sense can be made of the desuspension construction in such a way that the equivalence (*) exists.
\end{conjecture}

This conjecture is inspired by the corresponding algebraic result of Fresse \cite{fresse:2011}. In joint work with Paolo Salvatore, we have constructed an equivalence of symmetric sequences of the form (*), but not an equivalence of operads. As far as we know, the conjecture remains open.

\begin{conjecture} \label{conj:grav}
Let $\mathcal{M}_{0,n+1}$ denote the moduli space of genus $0$ Riemann surfaces with $n+1$ marked points. Let $\overline{\mathcal{M}}_{0,n+1}$ denote the Deligne-Mumford compactification of $\mathcal{M}_{0,n+1}$. (We can identify $\overline{\mathcal{M}}_{0,n+1}$ with the moduli spaces of stable nodal genus $0$ curves with $n+1$ marked points.) Then $\overline{\mathcal{M}}_{0,*+1}$ is an operad of unbased topological spaces and so has an associated stable operad
\[ \Sigma^\infty_+ \overline{\mathcal{M}}_{0,*+1}. \]

The conjecture is that there is an equivalence of operads
\[ K(\Sigma^\infty_+ \overline{\mathcal{M}}_{0,*+1}) \homeq s^{-1}(\mathcal{D}_2)_{bS^1} \]
where the right-hand side is a desuspension of the `transfer operad' for the $S^1$-action on the little $2$-discs operad constructed by Westerland \cite{westerland:2006}, and mentioned in \ref{ex:operads}(3).
\end{conjecture}

Again this conjecture is inspired by a corresponding algebraic result: see Getzler \cite[4.6]{getzler:1995}.

\section{Proof of Proposition \ref{prop:P-CBP-symseq}} \label{sec:P-CBP}

This section is largely a rewrite of section 20 of \cite{arone/ching:2009}. The reason for the rewrite is to prove a version of Theorem 20.2 of \cite{arone/ching:2009} that has no finiteness hypotheses. This more general version specializes to Proposition \ref{prop:P-CBP-symseq} which plays an important role in the proof of Theorem \ref{thm:W-CB}. In this section we make extensive use of two-sided bar and cobar constructions for modules and comodules over operads and cooperads. These two-sided constructions are described in detail in \cite{ching:2005a} and we use the terminology of \S7 of \cite{ching:2005a} without further reference. One change in notation is that we write $\tilde{\mathsf{T}}(I)$ for the category of generalized $I$-trees. This category was denoted by $\mathsf{Tree}(I)$ in \cite{ching:2005a} but that is too close to the notation of \S\ref{sec:precooperads}.

For a reduced operad $P$, let $R$ be a right $P$-module and $L$ a left $P$-module. The decomposition maps of \cite[7.18]{ching:2005a} associated to the two-sided bar construction determine a natural map
\[ \delta: B(R,P,L) \to C(B(R,P,\mathsf{1}),BP,B(\mathsf{1},P,L)) \]
where $\mathsf{1}$ is the unit symmetric sequence.

\begin{prop} \label{prop:CB}
The map $\delta$ is a weak equivalence when $R,P,L$ are all termwise-cofibrant.
\end{prop}

This is a generalization of Theorem 20.2 of \cite{arone/ching:2009} in which we have replaced the double Koszul dual with the cobar-bar construction and removed the finiteness hypotheses. Proposition \ref{prop:P-CBP-symseq} follows by taking $R = L = P$ for the following reasons. As symmetric sequences, we have
\[ CBP \isom C(\mathsf{1},BP,\mathsf{1}) \homeq C(B(P,P,\mathsf{1}),BP,B(\mathsf{1},P,P)) \]
The second equivalence here is induced by equivalences of $BP$-comodules $\mathsf{1} \weq B(P,P,\mathsf{1})$, and $\mathsf{1} \weq B(\mathsf{1},P,P)$. By \ref{prop:CB} the right-hand side above is equivalent to
\[ B(P,P,P) \homeq P \]
as a symmetric sequence.

We prove Proposition \ref{prop:CB} by the same method as in the proof of \cite[20.2]{arone/ching:2009}, that is, by induction on the truncation tower of the right $P$-module $R$. The first step is to show that each side of the map $\delta$ preserves homotopy-fibre sequences.

\begin{lemma} \label{lem:cobar-fibre}
Let $P$ be a reduced operad and $L$ a left $P$-module, both termwise-cofibrant. Let $R \to R' \to R''$ be a homotopy-fibre sequence of termwise-cofibrant right $P$-modules (i.e. for each finite set $I$, $R(I) \to R'(I) \to R''(I)$ is a homotopy-fibre sequence of spectra). Then each side of the map $\delta$ induces homotopy-fibre sequences of symmetric sequences, when applied to $R \to R' \to R''$ in the $R$-variable.
\end{lemma}
\begin{proof}
The bar construction preserves fibre sequences (which are the same as cofibre sequences) in its right module term by Lemma 20.5 of \cite{arone/ching:2009}. It is sufficient to show that the cobar construction preserves fibre sequences in the right comodule term. Recall from \cite{ching:2005a} that
\[ C(-,Q,L)(I) = \Map_{T \in \tilde{\mathsf{T}}(I)}(w(T)_+,(-,Q,L)(T)) \]
where $\tilde{\mathsf{T}}(I)$ is the category of generalized $I$-trees and $w(-)_+$ is a certain diagram of simplicial sets indexed by $\tilde{\mathsf{T}}(I)$. By a similar argument to that in the proof of Proposition \ref{prop:bar-cobar-homotopy}, the diagram $w(-)_+$ is Reedy cofibrant.

The functor $(-,Q,L)(I)$ takes homotopy-fibre sequences of right $Q$-comodules to homotopy-fibre sequences of diagrams $\tilde{\mathsf{T}}(I) \to \spectra$. All the morphisms in the Reedy category $\tilde{\mathsf{T}}(I)$ increase degree (where the degree is the number of vertices), so homotopy-fibre sequences in the Reedy model structure on these diagrams are just the objectwise homotopy-fibre sequences. The mapping spectrum construction $\Map_{T \in \tilde{\mathsf{T}}(I)}(w(T)_+,-)$ takes these to homotopy-fibre sequences of spectra, as required.
\end{proof}

This Lemma allows us to reduce to the case where $R$ is a trivial right $P$-module. We next analyze the cobar construction in the target of $\delta$ in that case.

\begin{definition}[Dual composition product] \label{def:cocomprod}
Given two symmetric sequences $A_0,A_1$, the \emph{dual composition product} $A_1 \; \hat{\circ} \; A_0$ is the symmetric sequence with
\[ (A_1 \; \hat{\circ} \; A_0)(I) := \prod_{I \epi J} A_1(J) \smsh \Smsh_{j \in J} A_0(I_j). \]
Strictly, the indexing here is over the set of \emph{isomorphism classes} of surjections from $I$ to another finite set $J$, where $f:I \epi J$ and $f':I \epi J'$ are isomorphic if there is a bijection $\sigma:J \to J'$ such that $\sigma f = f'$. This set of isomorphism classes is in one-to-one correspondence with the set of unordered partitions of $I$ into nonempty subsets.

Because the smash product does not commute with products in $\spectra$, the dual composition product $\; \hat{\circ} \;$ does not define an associative monoidal structure on the category of symmetric sequences. However, we can still define objects that play the role of iterations of $\; \hat{\circ} \;$. Given $A_0,\dots,A_n$, we define the \emph{iterated dual composition product} by taking
\[ (A_n \; \hat{\circ} \; A_{n-1} \; \hat{\circ} \; \dots \; \hat{\circ} \; A_1 \; \hat{\circ} \; A_0)(I) \]
to be
\[ \prod_{I \epi J^{(1)} \epi \dots \epi J^{(n)}} A_n(J^{(n)}) \smsh \Smsh_{j \in J^{(n)}} A_{n-1}(J^{(n-1)}_j) \smsh \dots \smsh \Smsh_{j \in J^{(1)}} A_0(I_j). \]
The product is indexed by isomorphism classes of sequences of surjections $I \epi J^{(1)} \epi \dots \epi J^{(n)}$ of finite sets.

The cobar construction $C(S,Q,N)$, for a right $Q$-comodule $S$ and left $Q$-comodule $N$, is isomorphic to the totalization of a cosimplicial symmetric sequence with $k$-cosimplices given by
\[ S \; \hat{\circ} \; \overbrace{Q \; \hat{\circ} \; \dots \; \hat{\circ} \; Q}^k \; \hat{\circ} \; N. \]
The coface maps are determined by the cooperad decomposition on $Q$ and the comodule structures on $S$ and $N$. The codegeneracy maps are determined by the counit map $Q \to \mathsf{1}$ for the cooperad $Q$. These are referred to in \cite[7.15]{ching:2005a} and are spelled out in more detail in \cite{ching:2005c}.
\end{definition}

\begin{lemma} \label{lem:cobar-trivial}
Let $P$ be a reduced operad, $R$ a trivial right $P$-module and $L$ any left $P$-module. Suppose that $R$, $P$ and $L$ are all termwise-cofibrant. Then there is a weak equivalence of symmetric sequences
\[ \delta': C(BR,BP,BL) \weq R \; \hat{\circ} \; BL \]
where $BR = B(R,P,\mathsf{1})$ and $BL = B(\mathsf{1},P,L)$.
\end{lemma}
\begin{proof}
We construct $\delta'$ in two steps. First, we claim that there is a levelwise weak equivalence of cosimplicial objects
\[ \tag{*} BR \; \hat{\circ} \; \overbrace{BP \; \hat{\circ} \; \dots \; \hat{\circ} \; BP}^k \; \hat{\circ} \; BL \weq R \; \hat{\circ} \; BP \; \hat{\circ} \; \overbrace{BP \; \hat{\circ} \; \dots \; \hat{\circ} \; BP}^k \; \hat{\circ} \; BL. \]
Note that this is not as simple as giving a map $BR \to R \; \hat{\circ} \; BL$ because the iterated dual composition product on the right-hand side is not obtained by iterating that on the left with another such product. Instead we have to define this term by term.

First note that
\[ BR = B(R,P,\mathsf{1}) \isom R \circ B(\mathsf{1},P,\mathsf{1}) = R \circ BP \]
because the right $P$-module structure on $R$ is trivial. The left-hand side of (*) is a product of terms of the form
\[ (R \circ BP)(J^{(k+1)}) \smsh BP(J^{(k)}_j) \smsh \dots \]
where $I \epi J^{(1)} \epi \dots \epi J^{(k+1)}$ is a sequence of surjections of finite sets.

Each of these terms is a coproduct indexed over surjections $J^{(k+1)} \epi J^{(k+2)}$ of terms of the form
\[ R(J^{(k+2)}) \smsh BP(J^{(k+1)}_j) \smsh \dots \]
The map (*) can then be defined using the canonical map from coproduct to product:
\[ \begin{diagram}
  \node{\left[\prod_{I \epi J^{(1)} \epi \dots \epi J^{(k+1)}} \left[ \coprod_{J^{(k+1)} \epi J^{(k+2)}} R(J^{(k+2)}) \smsh \dots \right]\right]} \arrow{s} \\
  \node{\left[\prod_{I \epi J^{(1)} \epi \dots \epi J^{(k+1)}} \left[\prod_{J^{(k+1)} \epi J^{(k+2)}} R(J^{(k+2)}) \smsh \dots \right] \right]}
\end{diagram} \]
since the target is now isomorphic to the iterated dual composition product with $k+3$ terms. The map (*) as defined is a weak equivalence because for spectra, the map from a finite coproduct to the corresponding finite product is a weak equivalence.

The weak equivalences (*) determine a weak equivalence on the totalizations:
\[ \delta'_1: C(BR,BP,BL) \to \Tot(R \; \hat{\circ} \; BP \; \hat{\circ} \; \overbrace{BP \; \hat{\circ} \; \dots \; \hat{\circ} \; BP}^{\bullet} \; \hat{\circ} \; BL). \]
The cosimplicial object on the right-hand side is coaugmented over $R \; \hat{\circ} \; BL$ and has extra codegeneracies that arise from the `extra' `$BP$' term. (See \cite[1.17]{arone/ching:2009}.) It follows that there is a weak equivalence
\[ \delta'_2: \Tot(R \; \hat{\circ} \; BP \; \hat{\circ} \; \overbrace{BP \; \hat{\circ} \; \dots \; \hat{\circ} \; BP}^{\bullet} \; \hat{\circ} \; BL) \weq R \; \hat{\circ} \; BL. \]
Taking $\delta' = \delta'_2 \delta'_1$ gives us the weak equivalence
\[ \delta': C(BR,BP,BL) \weq R \; \hat{\circ} \; BL \]
required of the Lemma.
\end{proof}

\begin{proof}[Proof of Proposition \ref{prop:CB}]
First we verify that
\[ \delta: B(R,P,L) \to C(BR,BP,BL) \]
is a weak equivalence for trivial right $P$-modules $R$. To see this, we compose $\delta$ with the weak equivalence $\delta'$ of Lemma \ref{lem:cobar-trivial}. Notice also that we have an isomorphism
\[ R \circ BL  \isom B(R,P,L). \]
With respect to this isomorphism, $\delta' \delta$ is the canonical weak equivalence
\[ R \circ BL \weq R \; \hat{\circ} \; BL \]
from a finite coproduct of spectra to the corresponding product. Since $\delta'$ is also a weak equivalence, it follows that $\delta$ is a weak equivalence.

Now consider an arbitrary termwise-cofibrant right $P$-module $R$. Recall that we have a `truncation tower' for $R$ (just as we had for an operad $P$ in the proof of Theorem \ref{thm:W-CB}). This consists of a collection of fibre sequences
\[ R_{=n} \to R_{\leq n} \to R_{\leq(n-1)} \]
of right $P$-modules where $R_{\leq n}(I)$ is trivial for $|I| > n$ and is equal to $R(I)$ otherwise. The $P$-module structure maps are either equal to those for $R$, or are trivial as appropriate. The fibres $R_{=n}$ are trivial right $P$-modules concentrated in terms $R_{=n}(I)$ where $|I| = n$. Applying $\delta$, we get a diagram
\[ \begin{diagram}
  \node{B(R_{=n},P,L)} \arrow{e} \arrow{s,l}{\sim} \node{B(R_{\leq n},P,L)} \arrow{e} \arrow{s} \node{B(R_{\leq(n-1)},P,L)} \arrow{s} \\
  \node{C(BR_{=n},BP,BL)} \arrow{e} \node{C(BR_{\leq n},BP,BL)} \arrow{e} \node{C(BR_{\leq(n-1)},BP,BL)}
\end{diagram} \]
The rows are homotopy-fibre sequences by \cite[20.5]{arone/ching:2009} and Lemma \ref{lem:cobar-fibre}. The left-hand vertical map is an equivalence because $R_{=n}$ is a trivial right $P$-module. By induction on $n$, it follows that all the vertical maps in such diagrams are equivalences. Finally, since the $I$-terms of each side of the map $\delta$ depend only on $R(J)$ for $|J| \leq |I|$, it follows that $\delta$ is an equivalence for all right modules $R$.
\end{proof}

\bibliographystyle{amsplain}
\bibliography{../mcching}

\end{document}

%% file: fragments.pstex_t
\begin{picture}(0,0)%
\includegraphics{fragments.ps}%
\end{picture}%
%
%
\setlength{\unitlength}{3947sp}%
\begingroup\makeatletter\ifx\SetFigFont\undefined%
\gdef\SetFigFont#1#2#3#4#5{%
  \reset@font\fontsize{#1}{#2pt}%
  \fontfamily{#3}\fontseries{#4}\fontshape{#5}%
  \selectfont}%
\fi\endgroup%
\begin{picture}(7036,2325)(564,-1923)
\put(579,231){\makebox(0,0)[lb]{\smash{{\SetFigFont{11}{13.2}{\rmdefault}{\mddefault}{\updefault}{\color[rgb]{0,0,0}$1$}%
}}}}
\put(879,231){\makebox(0,0)[lb]{\smash{{\SetFigFont{11}{13.2}{\rmdefault}{\mddefault}{\updefault}{\color[rgb]{0,0,0}$2$}%
}}}}
\put(1179,231){\makebox(0,0)[lb]{\smash{{\SetFigFont{11}{13.2}{\rmdefault}{\mddefault}{\updefault}{\color[rgb]{0,0,0}$3$}%
}}}}
\put(1606,239){\makebox(0,0)[lb]{\smash{{\SetFigFont{11}{13.2}{\rmdefault}{\mddefault}{\updefault}{\color[rgb]{0,0,0}$4$}%
}}}}
\put(2094,231){\makebox(0,0)[lb]{\smash{{\SetFigFont{11}{13.2}{\rmdefault}{\mddefault}{\updefault}{\color[rgb]{0,0,0}$5$}%
}}}}
\put(2638,223){\makebox(0,0)[lb]{\smash{{\SetFigFont{11}{13.2}{\rmdefault}{\mddefault}{\updefault}{\color[rgb]{0,0,0}$1$}%
}}}}
\put(3036,223){\makebox(0,0)[lb]{\smash{{\SetFigFont{11}{13.2}{\rmdefault}{\mddefault}{\updefault}{\color[rgb]{0,0,0}$2$}%
}}}}
\put(3448,223){\makebox(0,0)[lb]{\smash{{\SetFigFont{11}{13.2}{\rmdefault}{\mddefault}{\updefault}{\color[rgb]{0,0,0}$3$}%
}}}}
\put(3816,231){\makebox(0,0)[lb]{\smash{{\SetFigFont{11}{13.2}{\rmdefault}{\mddefault}{\updefault}{\color[rgb]{0,0,0}$4$}%
}}}}
\put(4183,231){\makebox(0,0)[lb]{\smash{{\SetFigFont{11}{13.2}{\rmdefault}{\mddefault}{\updefault}{\color[rgb]{0,0,0}$5$}%
}}}}
\put(736,-534){\makebox(0,0)[lb]{\smash{{\SetFigFont{11}{13.2}{\rmdefault}{\mddefault}{\updefault}{\color[rgb]{0,0,0}$t$}%
}}}}
\put(1426,-961){\makebox(0,0)[lb]{\smash{{\SetFigFont{11}{13.2}{\rmdefault}{\mddefault}{\updefault}{\color[rgb]{0,0,0}$t'$}%
}}}}
\put(1306,-1621){\makebox(0,0)[lb]{\smash{{\SetFigFont{11}{13.2}{\rmdefault}{\mddefault}{\updefault}{\color[rgb]{0,0,0}$T$}%
}}}}
\put(3616,-1606){\makebox(0,0)[lb]{\smash{{\SetFigFont{11}{13.2}{\rmdefault}{\mddefault}{\updefault}{\color[rgb]{0,0,0}$U$}%
}}}}
\put(5390,209){\makebox(0,0)[lb]{\smash{{\SetFigFont{11}{13.2}{\rmdefault}{\mddefault}{\updefault}{\color[rgb]{0,0,0}$1$}%
}}}}
\put(5870,224){\makebox(0,0)[lb]{\smash{{\SetFigFont{11}{13.2}{\rmdefault}{\mddefault}{\updefault}{\color[rgb]{0,0,0}$2$}%
}}}}
\put(6298,224){\makebox(0,0)[lb]{\smash{{\SetFigFont{11}{13.2}{\rmdefault}{\mddefault}{\updefault}{\color[rgb]{0,0,0}$3$}%
}}}}
\put(7040, 44){\makebox(0,0)[lb]{\smash{{\SetFigFont{11}{13.2}{\rmdefault}{\mddefault}{\updefault}{\color[rgb]{0,0,0}$4$}%
}}}}
\put(7565, 37){\makebox(0,0)[lb]{\smash{{\SetFigFont{11}{13.2}{\rmdefault}{\mddefault}{\updefault}{\color[rgb]{0,0,0}$5$}%
}}}}
\put(5930,-1103){\makebox(0,0)[lb]{\smash{{\SetFigFont{11}{13.2}{\rmdefault}{\mddefault}{\updefault}{\color[rgb]{0,0,0}$U_t$}%
}}}}
\put(7123,-1853){\makebox(0,0)[lb]{\smash{{\SetFigFont{11}{13.2}{\rmdefault}{\mddefault}{\updefault}{\color[rgb]{0,0,0}$U_{t'}$}%
}}}}
\put(1134,-556){\makebox(0,0)[lb]{\smash{{\SetFigFont{11}{13.2}{\rmdefault}{\mddefault}{\updefault}{\color[rgb]{0,0,0}$e$}%
}}}}
\put(6624,-466){\makebox(0,0)[lb]{\smash{{\SetFigFont{11}{13.2}{\rmdefault}{\mddefault}{\updefault}{\color[rgb]{0,0,0}$e$}%
}}}}
\put(2214,-691){\makebox(0,0)[lb]{\smash{{\SetFigFont{11}{13.2}{\rmdefault}{\mddefault}{\updefault}{\color[rgb]{0,0,0}$\leq$}%
}}}}
\end{picture}%

%% file: trees.pstex_t
\begin{picture}(0,0)%
\includegraphics{trees.ps}%
\end{picture}%
%
%
\setlength{\unitlength}{3947sp}%
\begingroup\makeatletter\ifx\SetFigFont\undefined%
\gdef\SetFigFont#1#2#3#4#5{%
  \reset@font\fontsize{#1}{#2pt}%
  \fontfamily{#3}\fontseries{#4}\fontshape{#5}%
  \selectfont}%
\fi\endgroup%
\begin{picture}(6112,1868)(654,-1436)
\put(669,261){\makebox(0,0)[lb]{\smash{{\SetFigFont{11}{13.2}{\rmdefault}{\mddefault}{\updefault}{\color[rgb]{0,0,0}$1$}%
}}}}
\put(1149,254){\makebox(0,0)[lb]{\smash{{\SetFigFont{11}{13.2}{\rmdefault}{\mddefault}{\updefault}{\color[rgb]{0,0,0}$2$}%
}}}}
\put(1126,-1156){\makebox(0,0)[lb]{\smash{{\SetFigFont{11}{13.2}{\rmdefault}{\mddefault}{\updefault}{\color[rgb]{0,0,0}$T$}%
}}}}
\put(1659,254){\makebox(0,0)[lb]{\smash{{\SetFigFont{11}{13.2}{\rmdefault}{\mddefault}{\updefault}{\color[rgb]{0,0,0}$i$}%
}}}}
\put(2791,254){\makebox(0,0)[lb]{\smash{{\SetFigFont{11}{13.2}{\rmdefault}{\mddefault}{\updefault}{\color[rgb]{0,0,0}$3$}%
}}}}
\put(3699,269){\makebox(0,0)[lb]{\smash{{\SetFigFont{11}{13.2}{\rmdefault}{\mddefault}{\updefault}{\color[rgb]{0,0,0}$4$}%
}}}}
\put(3249,-1134){\makebox(0,0)[lb]{\smash{{\SetFigFont{11}{13.2}{\rmdefault}{\mddefault}{\updefault}{\color[rgb]{0,0,0}$U$}%
}}}}
\put(5739,-1366){\makebox(0,0)[lb]{\smash{{\SetFigFont{11}{13.2}{\rmdefault}{\mddefault}{\updefault}{\color[rgb]{0,0,0}$T \cup_i U$}%
}}}}
\put(5251,246){\makebox(0,0)[lb]{\smash{{\SetFigFont{11}{13.2}{\rmdefault}{\mddefault}{\updefault}{\color[rgb]{0,0,0}$1$}%
}}}}
\put(5754,246){\makebox(0,0)[lb]{\smash{{\SetFigFont{11}{13.2}{\rmdefault}{\mddefault}{\updefault}{\color[rgb]{0,0,0}$2$}%
}}}}
\put(6249,254){\makebox(0,0)[lb]{\smash{{\SetFigFont{11}{13.2}{\rmdefault}{\mddefault}{\updefault}{\color[rgb]{0,0,0}$3$}%
}}}}
\put(6751,269){\makebox(0,0)[lb]{\smash{{\SetFigFont{11}{13.2}{\rmdefault}{\mddefault}{\updefault}{\color[rgb]{0,0,0}$4$}%
}}}}
\end{picture}%

%% file: theta.pstex_t
\begin{picture}(0,0)%
\includegraphics{theta.ps}%
\end{picture}%
%
%
\setlength{\unitlength}{3947sp}%
\begingroup\makeatletter\ifx\SetFigFont\undefined%
\gdef\SetFigFont#1#2#3#4#5{%
  \reset@font\fontsize{#1}{#2pt}%
  \fontfamily{#3}\fontseries{#4}\fontshape{#5}%
  \selectfont}%
\fi\endgroup%
\begin{picture}(7023,3331)(1258,-2960)
\put(3733,-385){\makebox(0,0)[lb]{\smash{{\SetFigFont{11}{13.2}{\rmdefault}{\mddefault}{\updefault}{\color[rgb]{0,0,0}$1$}%
}}}}
\put(4033,-385){\makebox(0,0)[lb]{\smash{{\SetFigFont{11}{13.2}{\rmdefault}{\mddefault}{\updefault}{\color[rgb]{0,0,0}$2$}%
}}}}
\put(4333,-385){\makebox(0,0)[lb]{\smash{{\SetFigFont{11}{13.2}{\rmdefault}{\mddefault}{\updefault}{\color[rgb]{0,0,0}$3$}%
}}}}
\put(4760,-377){\makebox(0,0)[lb]{\smash{{\SetFigFont{11}{13.2}{\rmdefault}{\mddefault}{\updefault}{\color[rgb]{0,0,0}$4$}%
}}}}
\put(5248,-385){\makebox(0,0)[lb]{\smash{{\SetFigFont{11}{13.2}{\rmdefault}{\mddefault}{\updefault}{\color[rgb]{0,0,0}$5$}%
}}}}
\put(1273,-394){\makebox(0,0)[lb]{\smash{{\SetFigFont{11}{13.2}{\rmdefault}{\mddefault}{\updefault}{\color[rgb]{0,0,0}$1$}%
}}}}
\put(1671,-394){\makebox(0,0)[lb]{\smash{{\SetFigFont{11}{13.2}{\rmdefault}{\mddefault}{\updefault}{\color[rgb]{0,0,0}$2$}%
}}}}
\put(2083,-394){\makebox(0,0)[lb]{\smash{{\SetFigFont{11}{13.2}{\rmdefault}{\mddefault}{\updefault}{\color[rgb]{0,0,0}$3$}%
}}}}
\put(2451,-386){\makebox(0,0)[lb]{\smash{{\SetFigFont{11}{13.2}{\rmdefault}{\mddefault}{\updefault}{\color[rgb]{0,0,0}$4$}%
}}}}
\put(2818,-386){\makebox(0,0)[lb]{\smash{{\SetFigFont{11}{13.2}{\rmdefault}{\mddefault}{\updefault}{\color[rgb]{0,0,0}$5$}%
}}}}
\put(6474,-73){\makebox(0,0)[lb]{\smash{{\SetFigFont{11}{13.2}{\rmdefault}{\mddefault}{\updefault}{\color[rgb]{0,0,0}$1$}%
}}}}
\put(6954,-58){\makebox(0,0)[lb]{\smash{{\SetFigFont{11}{13.2}{\rmdefault}{\mddefault}{\updefault}{\color[rgb]{0,0,0}$2$}%
}}}}
\put(7382,-58){\makebox(0,0)[lb]{\smash{{\SetFigFont{11}{13.2}{\rmdefault}{\mddefault}{\updefault}{\color[rgb]{0,0,0}$3$}%
}}}}
\put(7238,-1269){\makebox(0,0)[lb]{\smash{{\SetFigFont{11}{13.2}{\rmdefault}{\mddefault}{\updefault}{\color[rgb]{0,0,0}$4$}%
}}}}
\put(7763,-1261){\makebox(0,0)[lb]{\smash{{\SetFigFont{11}{13.2}{\rmdefault}{\mddefault}{\updefault}{\color[rgb]{0,0,0}$5$}%
}}}}
\put(6817,-1723){\makebox(0,0)[lb]{\smash{{\SetFigFont{11}{13.2}{\rmdefault}{\mddefault}{\updefault}{\color[rgb]{0,0,0}$e$}%
}}}}
\put(1614,-997){\makebox(0,0)[lb]{\smash{{\SetFigFont{9}{10.8}{\rmdefault}{\mddefault}{\updefault}{\color[rgb]{0,0,0}$t$}%
}}}}
\put(2011,-1364){\makebox(0,0)[lb]{\smash{{\SetFigFont{9}{10.8}{\rmdefault}{\mddefault}{\updefault}{\color[rgb]{0,0,0}$t'$}%
}}}}
\put(2491,-1224){\makebox(0,0)[lb]{\smash{{\SetFigFont{9}{10.8}{\rmdefault}{\mddefault}{\updefault}{\color[rgb]{0,0,0}$t''$}%
}}}}
\put(4139,-1313){\makebox(0,0)[lb]{\smash{{\SetFigFont{9}{10.8}{\rmdefault}{\mddefault}{\updefault}{\color[rgb]{0,0,0}$s$}%
}}}}
\put(4559,-1688){\makebox(0,0)[lb]{\smash{{\SetFigFont{9}{10.8}{\rmdefault}{\mddefault}{\updefault}{\color[rgb]{0,0,0}$s'$}%
}}}}
\put(7524,-2139){\makebox(0,0)[lb]{\smash{{\SetFigFont{9}{10.8}{\rmdefault}{\mddefault}{\updefault}{\color[rgb]{0,0,0}$t''$}%
}}}}
\put(7424,-2563){\makebox(0,0)[lb]{\smash{{\SetFigFont{9}{10.8}{\rmdefault}{\mddefault}{\updefault}{\color[rgb]{0,0,0}$1-s'$}%
}}}}
\put(7130,-898){\makebox(0,0)[lb]{\smash{{\SetFigFont{9}{10.8}{\rmdefault}{\mddefault}{\updefault}{\color[rgb]{0,0,0}$t'-s$}%
}}}}
\put(6845,-621){\makebox(0,0)[lb]{\smash{{\SetFigFont{9}{10.8}{\rmdefault}{\mddefault}{\updefault}{\color[rgb]{0,0,0}$t$}%
}}}}
\put(2132,-2270){\makebox(0,0)[lb]{\smash{{\SetFigFont{11}{13.2}{\rmdefault}{\mddefault}{\updefault}{\color[rgb]{0,0,0}$WP(I)$}%
}}}}
\put(4379,-2268){\makebox(0,0)[lb]{\smash{{\SetFigFont{11}{13.2}{\rmdefault}{\mddefault}{\updefault}{\color[rgb]{0,0,0}$\bar{w}(U)$}%
}}}}
\put(6346,-2664){\makebox(0,0)[lb]{\smash{{\SetFigFont{11}{13.2}{\rmdefault}{\mddefault}{\updefault}{\color[rgb]{0,0,0}$BP(U)$}%
}}}}
\end{picture}%

%% file: h.pstex_t
\begin{picture}(0,0)%
\includegraphics{h.ps}%
\end{picture}%
%
%
\setlength{\unitlength}{3947sp}%
\begingroup\makeatletter\ifx\SetFigFont\undefined%
\gdef\SetFigFont#1#2#3#4#5{%
  \reset@font\fontsize{#1}{#2pt}%
  \fontfamily{#3}\fontseries{#4}\fontshape{#5}%
  \selectfont}%
\fi\endgroup%
\begin{picture}(5276,3540)(1164,-3731)
\put(2019,-3316){\makebox(0,0)[lb]{\smash{{\SetFigFont{11}{13.2}{\rmdefault}{\mddefault}{\updefault}{\color[rgb]{0,0,0}$\Delta[1] \times \Delta[1]^{\mathsf{rev}}$}%
}}}}
\put(5281,-3339){\makebox(0,0)[lb]{\smash{{\SetFigFont{11}{13.2}{\rmdefault}{\mddefault}{\updefault}{\color[rgb]{0,0,0}$\Delta[1]$}%
}}}}
\put(1261,-354){\makebox(0,0)[lb]{\smash{{\SetFigFont{11}{13.2}{\rmdefault}{\mddefault}{\updefault}{\color[rgb]{0,0,0}$(0,1)$}%
}}}}
\put(3354,-354){\makebox(0,0)[lb]{\smash{{\SetFigFont{11}{13.2}{\rmdefault}{\mddefault}{\updefault}{\color[rgb]{0,0,0}$(1,1)$}%
}}}}
\put(2326,-3661){\makebox(0,0)[lb]{\smash{{\SetFigFont{11}{13.2}{\rmdefault}{\mddefault}{\updefault}{\color[rgb]{0,0,0}$(t,s)$}%
}}}}
\put(5004,-3646){\makebox(0,0)[lb]{\smash{{\SetFigFont{11}{13.2}{\rmdefault}{\mddefault}{\updefault}{\color[rgb]{0,0,0}$\max(t-s,0)$}%
}}}}
\put(1179,-2971){\makebox(0,0)[lb]{\smash{{\SetFigFont{11}{13.2}{\rmdefault}{\mddefault}{\updefault}{\color[rgb]{0,0,0}$(0,0)$}%
}}}}
\put(3325,-2851){\makebox(0,0)[lb]{\smash{{\SetFigFont{11}{13.2}{\rmdefault}{\mddefault}{\updefault}{\color[rgb]{0,0,0}$(1,0)$}%
}}}}
\put(4239,-1899){\makebox(0,0)[lb]{\smash{{\SetFigFont{11}{13.2}{\rmdefault}{\mddefault}{\updefault}{\color[rgb]{0,0,0}$0$}%
}}}}
\put(6385,-1896){\makebox(0,0)[lb]{\smash{{\SetFigFont{11}{13.2}{\rmdefault}{\mddefault}{\updefault}{\color[rgb]{0,0,0}$1$}%
}}}}
\end{picture}%